\documentclass[reqno,12pt,a4paper]{amsart}
\usepackage{amssymb,amsmath,amstext,amsthm,amsfonts,amscd,bbold,wasysym}
\usepackage{euscript,a4wide}
\usepackage{graphicx}
\usepackage[shortlabels]{enumitem} 

\usepackage[colorlinks=true,backref]{hyperref}

\theoremstyle{plain}
\newtheorem{maintheorem}{Theorem}

\newtheorem{maincorollary}[maintheorem]{Corollary}

\newtheorem{theorem}{Theorem}[section]
\newtheorem{proposition}[theorem]{Proposition}

\newtheorem{lemma}[theorem]{Lemma}
\newtheorem{claim}[theorem]{Claim}
\newtheorem{remark}[theorem]{Remark}

\theoremstyle{definition}

\newtheorem{conjecture}{Conjecture}

\newcommand{\RR}{{\mathbb R}}

\newcommand{\DD}{{\mathbb D}}

\newcommand{\CC}{{\mathbb C}}
\newcommand{\ZZ}{{\mathbb Z}}

\newcommand{\sS}{{\mathbb S}}

\def \TT {{\mathbb T}}

\newcommand{\fX}{{\mathfrak{X}}}

\newcommand{\cC}{{\mathcal C}}
\newcommand{\cF}{{\mathcal F}}

\newcommand{\cO}{{\mathcal O}}

\newcommand{\F}{\EuScript{F}}

\newcommand{\vfi}{{\varphi}}
\renewcommand{\epsilon}{\varepsilon}

\newcommand{\qand}{\quad\text{and}\quad}
\newcommand{\wh}{\widehat}
\newcommand{\wt}{\widetilde}

\DeclareMathOperator{\diag}{diag}

\DeclareMathOperator{\spec}{sp}

\DeclareMathOperator{\supp}{supp}

\DeclareMathOperator{\sing}{Sing}

\DeclareMathOperator{\m}{Leb}
\DeclareMathOperator{\close}{Closure}

\DeclareMathOperator{\diver}{div}

\def \Leb {\operatorname{Leb}}
\def \leb {\operatorname{Leb}}

  \def \vfi {\varphi}

\title[Physical measures for high co-dimensional ASH]{Physical measures
  for asymptotically sectional expanding flows in higher co-dimensions}

\date{\today}

\thanks{V.A. was partially supported by CNPq-Brazil (grant
  304047/2023-6). L.S. was partially supported by FAPERJ-Funda\c c\~ao
  Carlos Chagas Filho de Amparo \`a Pesquisa do Estado do Rio de
  Janeiro Projects APQ1-E-26/211.690/2021 SEI-260003/015270/2021 and
  JCNE-E-26/200.271/2023 SEI-260003/000640/2023, by Coordena\c c\~ao
  de Aperfei\c coamento de Pessoal de N\'ivel Superior CAPES -- Finance
  Code 001 and PROEXT-PG project Dynamic Women - Din\^amicas,
  CNPq-Brazil (grant Projeto Universal 404943/2023-3).}

\author{Vitor Araujo and Luciana Salgado}

\address[V.A.]{Instituto de Matem\'atica e Estat\'{\i}stica,
  Universidade Federal da Bahia, Av. Milton Santos s/n,
  40170-110 Salvador, Brazil.}
\email{vitor.araujo.im.ufba@gmail.com, vitor.d.araujo@ufba.br}
\urladdr{https:///sites.google.com/view/vitor-araujo-ime-ufba}

\address[L.S.]{Universidade Federal do Rio de Janeiro, Instituto de
   Matem\'atica\\
   Avenida Athos da Silveira Ramos 149 Cidade Universit\'aria, P.O. Box 68530, 
   21941-909 Rio de Janeiro-RJ-Brazil }
 \email{lsalgado@im.ufjr.br, lucianasalgado@ufrj.br}
 \urladdr{http://www.im.ufrj.br/~lsalgado}
 


\begin{document}

\begin{abstract}
  We obtain sufficient conditions for the existence of physical/SRB
  measures for asymptotically sectionally hyperbolic attracting sets with
  any finite co-dimension, extending the co-dimension two case.

  We provide examples of such attractors, either with non-sectional
  hyperbolic equilibria, or with sectional-hyperbolic equilibria of
  mixed type, i.e., with a Lorenz-like singularity together with a
  Rovella-like singularity in a transitive set. These are
  higher-dimensional versions of contracting Lorenz-like attractors
  (also known as Rovella-like attractors) to which we apply our
  criteria to obtain a physical/SRB measure with full ergodic basin.

  We also adapt the previous examples to obtain higher co-dimensional
  non-uniformly sectional expanding attractors; and also asymptotical
  $p$-sectional hyperbolic attractors which are \emph{not} non-uniformly
  $(p-1)$-expanding, for any finite $p>2$.
\end{abstract}
 
\maketitle

\tableofcontents

\section{Introduction}
\label{sec:introduction}

The theory of uniformly hyperbolic systems, since its inception with
the seminal work of Smale, Anosov, and Sinai, has provided a robust
framework for understanding complex dynamical behavior, including the
existence of physical (or SRB) measures for hyperbolic
attractors~\cite{AS67,Sm67,BR75}. However, many important dynamical
systems arising in applications --- such as the Lorenz attractor ---
are not uniformly hyperbolic, yet exhibit rich and persistent chaotic
dynamics; see e.g.~\cite{AraPac2010s}.

Labarca and Pacifico~\cite{LP86} introduced the singular horseshoe, a
variation the geometric Lorenz attractor conceived to disprove
Palis-Smale's stability conjecture for flows on manifolds with
boundary. Later, Rovella~\cite{Ro93} introduced another variation of
geometric Lorenz attractor, replacing the singularity by one with a
central contracting condition. These models are known as contracting
Lorenz models or simply Rovella attractors, and their singularities
known as ``Rovella-like''.

A fundamental breakthrough was achieved by
Morales-Pacifico-Pujals~\cite{MPP99} with the concept
of~\emph{singular hyperbolicity}, which captures a weaker form of
hyperbolicity compatible with the presence of equilibria. This allows,
for instance, a rigorous description of the Lorenz
attractor~\cite{MPP04,Tu99}. In its original form, singular
hyperbolicity requires a dominated splitting
\( T_\Lambda M = E^s \oplus E^c \) into a uniformly contracting
subbundle \( E^s \) and a \emph{volume-expanding} central subbundle
\( E^c \). A strengthening of this notion, \emph{sectional
  hyperbolicity}, demands that every $2$-plane inside \( E^c \) is
expanded in area, which guarantees the existence of physical/SRB
measures for the attracting
sets~\cite{MeMor08,APPV,LeplYa17,araujo_2021}, which encompass the
multidimensional Lorenz attractor~\cite{BPV97}.

A \emph{physical measure} is an invariant probability measure for
which time averages exist and coincide with the space average, for a
set of initial conditions with positive Lebesgue measure, i.e. in the
weak$^*$ topology of convergence of probability measures we have
$$
B(\mu):=\left\{z\in M: \lim_{T\nearrow\infty} \frac1T \int_0^T
  \delta_{\phi_t (z)} \,dt = \mu \right\}
\;\text{with  } \Leb(B(\mu))>0.
$$
This set is the \emph{basin} of the measure. Sinai, Ruelle and Bowen
introduced this notion about fifty years ago, and proved that, for
uniformly hyperbolic (Axiom~A) diffeomorphisms and flows, time averages
exist for Lebesgue almost every point and coincide with one of
finitely many physical measures; see~\cite{BR75,Si72}.

Recent developments have extended these ideas in several
directions. One of the coauthors~\cite{Salgado19} introduced the
concept of $p$-sectional hyperbolicity, where the central bundle is
required to be $p$-sectionally expanding --- that is, the
$p$-dimensional volume is uniformly expanded along the central
direction.

The study of \emph{asymptotically sectional-hyperbolic sets},
introduced by Morales and San Martin~\cite{MorSM17}, recently advanced
by San Martin and Vivas~\cite{smvivas,SmartinVivas20}, extended the
theory encompassing systems where hyperbolicity holds asymptotically
outside the stable manifolds of singularities, including attractors
with Rovella-like singularities in any three-dimensional manifold, and
the singular-horseshoe

One of the coauthors with Castro, Pacifico and Pinheiro~\cite{ACPP11}
proposed a multidimensional analogue of the Rovella attractor,
featuring singularities with multidimensional expanding directions and
physical measures supported on non-uniformly expanding
attractors. More recently, together with
Sousa~\cite{ArSal25}, we established conditions for the existence of
physical/SRB measures in partially hyperbolic attracting sets with
non-uniform sectional expansion, and for asymptotically sectional
hyperbolic attracting sets with two-dimensional central direction ---
the ``co-dimension two'' case, unifying known examples like
Lorenz and Rovella attractors.

Nevertheless, most existing results focus on low codimensions or
require strong non-uniformity conditions. A natural and important
question is whether these results can be extended to attractors with
\emph{any finite co-dimension} (that is, any finite central dimension)
and more diverse singularity types, including \emph{not} sectionally
hyperbolic singularities; or exhibit \emph{mixed-type} singularities,
i.e., coexisting Lorenz-like and Rovella-like singularities.

In this work, we obtain sufficient conditions for the existence of
physical/SRB measures for asymptotic sectionally hyperbolic attracting
sets with any finite co-dimension, thus generalizing the known results
for codimension two, allowing us to handle attracting sets with slow
recurrence to equilibria and weak asymptotic sectional expansion.

We construct new examples of higher co-dimensional attractors that are
asymptotically sectionally hyperbolic and either contain
non-sectionally hyperbolic equilibria; or combine Lorenz-like and
Rovella-like singularities in the same transitive set; to which we
apply our existence result for physical/SRB measures with full ergodic
basins.

Moreover, we adapt these constructions to produce: attractors with
non-uniformly sectionally expanding central direction in higher
co-dimension; and asymptotic $p$-sectional hyperbolic attractors that
are \emph{bot} non-uniformly $(p-1)$-expanding along the central
direction, for any given dimension $p > 2$.

These examples show that a theory of $p$-sectional hyperbolicity and
asymptotic sectional hyperbolicity is not only natural but essential
for describing dynamics beyond the low-codimension regime, and that
physical measures are present even among mixed singularity
configurations.

\subsection{Statements of the results}
\label{sec:statements-results}

Let $M$ be a compact connected manifold with dimension $\dim M=m$,
endowed with a Riemannian metric, induced distance $d$ and volume form
$\m$. Let $\fX^r(M)$, $r\ge1$, be the set of $C^r$ vector fields on $M$
endowed with the $C^r$ topology and denote by $\phi_t$ the flow
generated by $G\in\fX^r(M)$.

\subsubsection{Preliminary definitions}
\label{sec:prelim-definit}

We say that $\sigma\in M$ with $G(\sigma)=0$ is an {\em equilibrium}
or \emph{singularity}. In what follows we denote by $\sing(G)$ the
family of all such points. We say that an equilibrium
$\sigma\in\sing(G)$ is \emph{hyperbolic} if all the eigenvalues of
$DG(\sigma)$ have non-zero real part.

An \emph{invariant set} $\Lambda$ for the flow $\phi_t$, generated by
the vector field $G$, is a subset of $M$ which satisfies
$\phi_t(\Lambda)=\Lambda$ for all $t\in\RR$. A point $p\in M$ is
\emph{periodic} for the flow $\phi_t$ generated by $G$ if $G(p)\neq 0$
and there exists $\tau>0$ so that $\phi_\tau(p)=p$; its orbit
$\cO_G(p)=\phi_{\RR}(p)=\phi_{[0,\tau]}(p)=\{\phi_tp: t\in[0,\tau]\}$
is a \emph{periodic orbit}, an invariant simple closed curve for the
flow.  An invariant set is \emph{nontrivial} if it is not a finite
collection of periodic orbits and equilibria.

Given a compact invariant set $\Lambda$ for $G\in \fX^r(M)$, we say
that $\Lambda$ is \emph{isolated} if there exists an open set
$U\supset \Lambda$ such that
$ \Lambda =\bigcap_{t\in\RR}\close{\phi_t(U)}$.  If $U$ can be chosen
so that $\close{\phi_t(U)}\subset U$ for all $t>0$, then we say that
$\Lambda$ is an \emph{attracting set} and $U$ a \emph{trapping region}
(or \emph{isolating neighborhood}) for
$\Lambda=\Lambda_G(U)=\cap_{t>0}\close{\phi_t(U)}$.

An \emph{attractor} is a transitive attracting set, that is,
an attracting set $\Lambda$ with a point $z\in\Lambda$ so that its
$\omega$-limit
$
  \omega(z):=\left\{y\in M: \exists t_n\nearrow\infty\text{  s.t.
  } \phi_{t_n}z\xrightarrow[n\to\infty]{}y \right\}
$
coincides with $\Lambda$.

\subsubsection{Partial hyperbolic attracting sets for vector fields}
\label{sec:part-hyperb-diff}

Let $\Lambda$ be a compact invariant set for $G \in \fX^r(M)$.  We say
that $\Lambda$ is {\em partially hyperbolic} if the tangent bundle
over $\Lambda$ can be written as a continuous $D\phi_t$-invariant
Whitney sum $ T_\Lambda M=E^s\oplus E^{cu}, $ where
$d_s=\dim(E^s_x)\ge1$ and $d_{cu}=\dim (E^{cu}_x)\ge2$ for $x\in\Lambda$,
and there exists a constant $\lambda >0$ such that for all
$x \in \Lambda$, $t\ge0$, we have
\begin{itemize}
  \item domination of the splitting:
$\|D\phi_t | E^s_x\| \cdot \|D\phi_{-t} | E^{cu}_{\phi_tx}\| \le e^{-\lambda t}$;
\item uniform contraction along $E^s$:
  $\|D\phi_t | E^s_x\| \le e^{-\lambda t}$;
\end{itemize}
for some choice of the Riemannian metric on the manifold, see
e.g. \cite{Goum07}. Changing the metric does not change the rate
$\lambda$ but might introduce the multiplication by a constant.

Then $E^s$ is the stable bundle and $E^{cu}$ the center-unstable
bundle.

\begin{remark}[domination and partial hyperbolicity for vector fields]
  \label{rmk:domparthyp}
  In the vector field setting, a dominated splitting is automatically
  partially hyperbolic whenever the flow direction is contained in the
  central-unstable bundle $X\in E^{cu}$.  In fact, this inclusion is
  equivalent to partial hyperbolicity; see e.g.~\cite[Lemma
  3.2]{arsal2012a}.  Since the flow direction is invariant, partial
  hyperbolicity is the natural setting to consider when studying
  invariant sets (which are not composed only of equilibria) for flows
  with a dominated splitting.
\end{remark}

A {\em partially hyperbolic attracting set} is a partially hyperbolic
set that is also an attracting set.

\subsubsection{Singular/sectional-hyperbolicity}
\label{sec:singul-hyperb-asympt}

The center-unstable bundle $E^{cu}$ is \emph{volume expanding} if
there exists $K,\theta>0$ such that
$|\det(D\phi_t| E^{cu}_x)|\geq K e^{\theta t}$ for all $x\in \Lambda$,
$t\geq 0$.

We say that a compact nontrivial invariant set $\Lambda$ is a
\emph{singular hyperbolic set} if all equilibria in $\Lambda$ are
hyperbolic, and $\Lambda$ is partially hyperbolic with volume
expanding center-unstable bundle.  A singular hyperbolic set which is
also an attracting set is called a {\em singular hyperbolic attracting
  set}.

For any given $2\le p\le d_{cu}:=\dim E^{cu}$, we say that $E^{cu}$ is
\emph{$p$-sectionally expanding} if there are
positive constants $K , \theta$ such that for every $x \in \Lambda$
and every $p$-dimensional linear subspace $L_x \subset E^{cu}_x$ one has
$|\det(D\phi_t| L_x)|\geq K e^{\theta t}$ for all $t\geq 0$. 

A \emph{$p$-sectional-hyperbolic (attracting) set} is a partially
hyperbolic (attracting) set whose central subbundle is $p$-sectionally
expanding.

The case $p=2$ is simply denoted \emph{sectionally expanding} and
\emph{sectional-hyperbolicity} respectively, in what follows.

\subsubsection{Asymptotical sectional-hyperbolicity}
\label{sec:asympt-singul-hyperb}

A compact invariant partially hyperbolic set $\Lambda$ of a vector
field $G$ whose equilibria are hyperbolic, is \emph{asymptotically
  sectional-hyperbolic} (ASH) if the center-unstable subbundle is
eventually asymptotically sectional expanding outside the stable
manifold of the equilibria. That is, there exists $c_*>0$ so that the
\emph{asymptotically expanding condition} (ASE) holds
  \begin{align}\label{eq:assecexp}
    \limsup_{T\nearrow\infty}\frac1T \log|\det(D\phi_T\mid_{F_x})|\ge c_*
  \end{align}
  for every
  $x\in\Lambda\setminus \cup\{W^s_\sigma:\sigma\in\sing_\Lambda(G)\}$
  and each $2$-dimensional linear subspace $F_x$ of $E^{cu}_x$, where
  we write $\sing_\Lambda(G)=\sing(G)\cap\Lambda$ and
  $W^s_\sigma=\{x\in M: \lim_{t\to+\infty}\phi_tx = \sigma\}$ is the
  \emph{stable manifold} of the hyperbolic equilibrium $\sigma$. It is
  well-known that $W^s_\sigma$ is a immersed submanifold of $M$; see
  e.g.~\cite{PM82}. This implies that all transverse directions to the
  vector field along the center-unstable subbundle have positive
  Lyapunov exponent; that is, if $v\in E^{cu}_x\setminus(\RR\cdot G)$,
  then
  $\chi(x,v):=\limsup_{t\to+\infty}\log\|D\phi_t(x)v\|^{1/t}\ge
  c_*>0$; see~\cite[Theorem 1.6]{ArSal25}.
  
  \begin{lemma}[Hyperbolic Lemma]
\label{le:hyplemma}
Every compact invariant subset $\Gamma$ without equilibria contained
in a asymptotically sectional-hyperbolic set is uniformly hyperbolic.
\end{lemma}
  \begin{proof}
    See e.g.~\cite[Proposition 1.8]{MPP04} for sectional-hyperbolic
    sets; and~\cite[Theorem 2.2]{SmartinVivas20} for the
    asymptotically sectional-hyperbolic case.
  \end{proof}

  We say that an invariant compact subset $\Gamma$ is
  \emph{(uniformly) hyperbolic} if $\Gamma$ is partially hyperbolic
  and the central-unstable bundle admits a continuous splitting
  $E^{cu}=(\RR\cdot G)\oplus E^u$, with $\RR\cdot G$ the
  one-dimensional invariant flow direction and $E^u$ a uniformly
  expanding subbundle. That is, we get the following dominated
  splitting $T_\Gamma M= E^s\oplus(\RR\cdot G)\oplus E^u$ into
  three-subbundles; see e.g.~\cite{fisherHasselblatt12}.

\subsubsection{Asymptotical $p$-sectional hyperbolicity}
\label{sec:asympt-p-section}

Analogously, we say that a compact invariant partially hyperbolic set
$\Lambda$ of a vector field $G$ whose equilibria are hyperbolic, is
\emph{asymptotically $p$ sectional-hyperbolic} (pASH) if the
center-unstable subbundle is eventually asymptotically $p$-sectional
expanding outside the stable manifold of the equilibria: that is there
exists $c_*>0$ so that~\eqref{eq:assecexp} holds for every
$x\in\Lambda\setminus \cup\{W^s_\sigma:\sigma\in\sing_\Lambda(G)\}$
and replacing $F_x$ by any $p$-dimensional linear subspace of
$E^{cu}_x$.

Here it is implicitly assumed that $2\le p\le d_{cu}$ is fixed.
  
\subsection{Non-uniform sectional expansion}
\label{sec:non-uniformly-expand-1}

Let us fix $G\in\fX^2(M)$ endowed with a partially hyperbolic
attracting set $\Lambda=\Lambda_G(U)$ with a trapping region $U$.
Then we can take a continuous extension
$T_UM=\wt{E^s}\oplus \wt{E^{cu}}$ of $T_\Lambda M=E^s\oplus E^{cu}$
and for small $a>0$ find center unstable and stable cones
\begin{align}
  \label{eq:cucone}
  C^{cu}_a(x)
  &=
    \{ v=v^s+v^c : v^s\in \wt{E^s}_x, v^c\in\wt{E^{cu}}_x, x\in
    U, \|v^s\|\le a\|v^c\|\}, \qand
  \\
  C^s_a(x)
  &=
    \{ v=v^s+v^c : v^s\in \wt{E^s}_x, v^c\in\wt{E^{cu}}_x, x\in
    U, \|v^c\|\le a\|v^s\|\}, \nonumber
\end{align}
which are invariant in the following sense
\begin{align}\label{eq:coneinv}
  D\phi_t(x)\cdot C^{cu}_a(x)\subset C^{cu}_a(\phi_t(x))
  \qand
  D\phi_{-t}\cdot C^{s}_a(x)\supset C^{s}_a(\phi_{-t}(x)),
\end{align}
for all $x\in\Lambda$ and $t>0$ so that $\phi_{-s}(x)\in U$ for all
$0<s\le t$; see e.g.~\cite{ArMel17}. We can assume
that $\wt{E^{cu}_x}\subset C^{cu}_a(x)$ still contains the flow
direction $G(x)$ for each $x\in U$; see~\cite[Section 1.2]{ArSal25}.
 We can
also assume, without loss of generality according to~\cite{ArMel17},
that the continuous extension of the stable direction $E^s$ of the
splitting is still $D\phi_t$-invariant and
$\wt{E^s_x}\subset C^s_a(x), x\in U$.
In what follows, we keep the notation $T_U M=E^s\oplus E^{cu}$ and
write $N_x^{cu}=E^{cu}_x\cap G(x)^{\perp}, x\in U$.

We write $f:=\phi_1$ for the time-$1$ diffeomorphism
induced by the flow. We say that the attracting set $\Lambda$ is
\begin{description}
\item[weak non-uniform $2$-sectionally expanding (wNU2SE)] if
  there exists $c_0>0$ so that
  \begin{align}\label{eq:wNU2SE}
    \Omega=\left\{
    x\in U: \liminf_{n\nearrow\infty}\frac1n\sum\nolimits_{i=0}^{n-1}
    \log
    \big\|\wedge^2 (Df\mid_{E^{cu}_{f^i x}})^{-1}\big\|
    \le - c_0
    \right\}
    \qand \m(\Omega)>0.
  \end{align}
\end{description}
This is enough to ensure existence a physical/SRB measure under a slow
recurrence condition, as explained in what follows.

\subsection{Existence of physical/SRB measures}
\label{sec:existence-physic-mea}

We can ensure existence of an ergodic
physical/SRB measure if the partially hyperbolic splitting is of
codimension $2$.

\begin{theorem}{\cite[Theorem E]{ArSal25}}
  \label{thm:fABV}
  Let a partially hyperbolic attracting set $\Lambda=\Lambda_G(U)$ for
  a vector field $G\in\fX^2(M)$ be given, with $d_{cu}=\dim
  E^{cu}=2$. Then $\Lambda$ satisfies the following weak asymptotical
  sectional expanding (wASE) condition on a positive volume
  subset
  \begin{align}\label{eq:wASE}
    \leb\left( \left\{ x\in U: \liminf_{T\nearrow\infty} \frac1T
    \log|\det (D\phi_T\mid_{E^{cu}_x})| >0 \right\} \right) >0
  \end{align}
  \emph{if, and only if}, there exists a
  physical/SRB ergodic hyperbolic measure $\mu$.
  If $\Lambda$ is transitive, then $\mu$ is unique and
  $\leb(\Omega\setminus B(\mu))=0$.

  Reciprocally, without restriction on $d_{cu}$, the existence of an
  invariant ergodic hyperbolic physical/SRB measure implies that
  (wNU2SE) holds on a positive volume subset of $U$.
\end{theorem}
Hence, to obtain a physical measure, it is enough to obtain a sequence
of times with asymptotic sectional expansion, along the trajectories
on a positive volume subset.

\begin{remark}[wNU2SE implies wASE]
  From~\cite[Theorem 1.6]{ArSal25} we have that a trajectory not
  converging to any singularity and satisfying condition wNU2SE,
  also satisfies wASE. Thus, we can replace~\eqref{eq:wASE}
  by~\eqref{eq:NU2SE} keeping the conclusion of
  Theorem~\ref{thm:fABV}.
\end{remark}

We can also show the existence of a physical probability measure for
weak ASH attracting sets.
\begin{theorem}[Physical/SRB measure for weak ASH attracting
  sets]{\cite[Corollary G]{ArSal25}}
  \label{thm:physASH}
  Let a $C^2$ vector field $G$ on $M$ and a trapping region $U$ be
  given containing a partially hyperbolic attracting set
  $\Lambda=\Lambda_G(U)$ with $d_{cu}=2$ so that every $x\in\Lambda$
  not converging to any equilibrium satisfies the wASE
  condition~\eqref{eq:wASE}.

  If $\Lambda$ contains only saddle-type hyperbolic equilibria, then
  there exists a physical/SRB probability measure supported on
  $\Lambda$.  If $\Lambda$ is transitive, then $\Lambda$ supports a
  unique physical/SRB probability measure whose basin covers a
  neighborhood of $\Lambda$.
\end{theorem}

To construct the physical/SRB measure in the presence of equilibria
for a partially hyperbolic attracting set in higher codimensions
(i.e. $d_{cu}>2$), the known proof requires control of the recurrence
near the equilibria, together with a strong form of the condition
wNU2SE.

Let $U\subset M$ be a forward invariant set of $M$ for the flow of a
$C^1$ vector field $G$, where all equilibria are hyperbolic.  We say
that the attracting set $\Lambda=\Lambda_G(U)$ is
\begin{description}
\item[non-uniform $2$-sectionally expanding (NU2SE)] if there exists
$c_0>0$ so that
  \begin{align}\label{eq:NU2SE}
    \Omega=\left\{
    x\in U: \limsup_{n\nearrow\infty}\frac1n\sum\nolimits_{i=0}^{n-1}
    \log
    \big\|\wedge^2 (Df\mid_{E^{cu}_{f^i x}})^{-1}\big\|
    \le - c_0
    \right\}
    \qand \m(\Omega)>0.
  \end{align}
\end{description}
We say that $G$ has
\begin{description}
\item[slow recurrence (SR)] if, on a positive Lebesgue measure subset
  $\Omega\subset U$, for every $\epsilon>0$, we can find $\delta>0$ so
  that
\begin{align}
  \label{eq:SR}
  \limsup_{n\nearrow\infty}\frac1n\sum\nolimits_{i=0}^{n-1}-\log d_\delta
  \big(f^i(x),\sing_\Lambda(G)\big) <\epsilon, \quad x\in\Omega;
\end{align}
\end{description}
where $d_\delta(x,S)$ $\delta$-{\em truncated distance} from $x\in M$
to a subset $S$, that is
\begin{align*}
  d_{\delta}(x,S)=\left\{
  \begin{array}{lll}
d(x,S) & \textrm{if $0<d(x,S)\leq\delta$;}
\\
\left(\frac{1-\delta}{\delta}\right)d(x,S)+2\delta-1 
& \textrm{if $\delta<d(x,S)<2\delta$;}
\\
1 & \textrm{if $d(x,S)\geq2\delta$.}
\end{array} \right.
\end{align*}

\begin{theorem}[Physical/SRB measures for non-uniformly
  sectionally expanding flows]{\cite[Theorem B]{ArSal25}}
  \label{thm:discretefabv}
  Let $G\in\fX^2(M)$ be a vector field with a partially hyperbolic
  attracting set $\Lambda=\Lambda_G(U)$ satisfying (SR) on
  $\Omega\subset U$, with $\m(\Omega)>0$.  Then we have NU2SE on
  $\Omega$ if, and only if, there are finitely many ergodic
  physical/SRB measures whose basins cover $\m$-a.e. point of
  $\Omega$:
  $\m\Big(\Omega\setminus \big(B(\mu_1)\cup\dots\cup
  B(\mu_p)\big)\Big)=0$.
\end{theorem}

\begin{remark}[NU2SE implies ASE]
  \label{rmk:nu2se-ase}
  According to \cite[Theorem 1.6]{ArSal25}, each trajectory not
  converging to a singularity and satisfying condition (NU2SE), also
  satisfies (ASE).
\end{remark}

\subsection{Physical measures for higher co-dimensional weak ASH
  attracting sets}
\label{sec:lorenz-like-rovella}

We obtain a necessary and sufficient condition for existence of a
physical/SRB measure for wASH attracting sets with arbitrary
codimension.

Let $U\subset M$ be a forward invariant set of $M$ for the flow of a
$C^1$ vector field $G$, where all equilibria are hyperbolic. We say
that $G$ has
  \begin{description}
  \item[weak slow recurrence (wSR)] if, on the positive Lebesgue
    measure subset $\Omega\subset U$, for every $\epsilon>0$, we can
    find $r>0$ so that
\begin{align}
  \label{eq:wSR}
  \limsup_{n\nearrow\infty}\frac1n\sum\nolimits_{i=0}^{n-1}
  \delta_{f^ix}(B_r(\sing_\Lambda(G)))<\epsilon, \quad x\in\Omega.
\end{align}
  \end{description}

  Clearly the slow recurrence (SR) condition implies the weak slow
  recurrence (wSR) condition.
  
  \begin{remark}[no atoms at equilibria]
    \label{rmk:noatom}
    If $x\in U$ satisfies~\eqref{eq:wSR}, then any $f$-invariant
    probability measure $\mu$ obtained as a weak$^*$ accumulation of the
    empirical measures $\Big(\frac1n\sum_{i=0}^{n-1}
    \delta_{f^ix}\Big)_{n\ge1}$ does not admit the elements of
    $\sing_\Lambda(G)$ as atoms: $\mu(\sing_\Lambda(G))=0$.
  \end{remark}

  \begin{maintheorem}[Physical measure for higher co-dimensional weak ASH]
    \label{mthm:hdASH}
    Let a $C^2$ vector field $G$ on $M$ and a trapping region $U$ be
    given containing a partially hyperbolic attracting set
    $\Lambda=\Lambda_G(U)$ so that every $x\in\Lambda$ not converging
    to any equilibrium satisfies wNU2SE.  We assume that $\Lambda$
    contains only saddle-type hyperbolic equilibria.

    If either one of the next condition holds
    \begin{enumerate}[(A)]
    \item $|\det(Df\mid_{E^{cu}_\sigma})|\ge1$ for all
      $\sigma\in\sing_\Lambda(G)$; or
    \item there exists a positive volume subset $\Omega\subset U$ so
      that $x\in\Omega$ satisfies (wSR);
    \end{enumerate}
  then there exists a physical/SRB  measure supported on
  $\Lambda$.

  If, in addition,  $\Lambda$ is transitive, then $\Lambda$ supports a
  unique physical/SRB probability measure whose basin covers a
  neighborhood of $\Lambda$.
\end{maintheorem}

Since a physical/SRB measure $\mu$ supported on $\Lambda$ cannot have
atoms by definition, then we deduce the following.

\begin{maincorollary}[existence of physical/SRB and weak slow recurrence]
  \label{mcor:wSR-SRB}
  In the same setting of Theorem~\ref{mthm:hdASH}, there exists a
  physical/SRB measure $\mu$ supported on $\Lambda$ if, and only if,
  weak slow recurrence holds on a positive volume subset
  $\Omega\supset B(\mu)$.
\end{maincorollary}

\subsection{New examples of attractors}
\label{sec:new-exampl-attract}

We distinguish between the following types of hyperbolic singularities
for flows.

\subsubsection{Generalized Lorenz-like singularities}
\label{sec:lorenz-like-singul}

We say that a singularity $\sigma$ belonging to a partially hyperbolic
set is \emph{generalized Lorenz-like} if $DG\mid_{E^{cu}_\sigma}$ has
a real eigenvalue $\lambda^s$ and
$\lambda^u=\inf\{\Re(\lambda):\lambda\in\spec(DG\mid_{E^{cu}_\sigma}),
\Re(\lambda)\ge0\}$ satisfies $-\lambda^u<\lambda^s<0<\lambda^u$.

This is a natural condition for singularities contained in partially
hyperbolic sets for flows, because of the following.

\begin{proposition}{\cite[Proposition 2.1]{araujo_2021}}
  \label{prop:generaLorenzlike} Let
  $\Lambda$ be a sectional hyperbolic attracting set and let
  $\sigma\in\Lambda$ be an equilibrium.
  If there exists $x\in\Lambda\setminus\{\sigma\}$ so that
  $\sigma\in\omega(x)\cup\alpha(x)$, then $\sigma$ is
  \emph{generalized Lorenz-like}.
\end{proposition}

We note that generalized Lorenz-like singularities are
sectional-expanding by definition, since $\lambda^s+\lambda^u>0$.

\subsubsection{Generalized Rovella-like singularities}
\label{sec:rovella-like-singul}

In the contracting Lorenz attractor, also known as Rovella attractor,
the singularity is sectionally contracting; see~\cite{Ro93}. We
extend this property to a partially hyperbolic setting with any
central-unstable dimension as follows.

We say that a singularity $\sigma$ belonging to a partially hyperbolic
set is \emph{generalized Rovella-like} if $DG\mid_{E^{cu}_\sigma}$ has
a real eigenvalue $\lambda^s$ and
$\lambda^u=\inf\{\Re(\lambda):\lambda\in\spec(DG\mid_{E^{cu}_\sigma}),
\Re(\lambda)\ge0\}$ satisfies $\lambda^s<-\lambda^u<0$.

\begin{remark}[stable index of the singularties]
  \label{rmk:stabindex}
  In both generalized Lorenz-like or Rovella-like cases, it is clear
  that the (stable) index of $\sigma$ is $\dim E^s_\sigma=d_s+1$.
\end{remark}

\subsubsection{Description of the examples}
\label{sec:description-examples}

We adapt the construction of the multidimensional Lorenz attractor,
first presented by Bonatti, Pumari\~no and Viana in~\cite{BPV97}, to
obtain the following examples:
\begin{itemize}
\item in Subsection~\ref{sec:multid-asympt-sectio}: an ASH attractor
  with $d_{cu}=2$ and non-sectional hyperbolic (``sectionally
  neutral'') equilibria type; and
\item in Subsection~\ref{sec:multid-ash-attract}: we adapt the
  previous example to obtain an ASH attractor with equilibria of mixed
  type: both Lorenz-like (sectionally expanding) and Rovella-like
  (sectionally contracting) equilibria in a transitive set.
\item in Subsection~\ref{sec:multidimesional-ash}: we extend the
  previous example to obtain ASH attractors with \emph{any given
    central-unstable dimension $d_{cu}>2$} and a pair of equilibria of
  either non-sectional-hyperbolic type, or mixed type.
\item in Subsection~\ref{sec:multid-ash-}: we obtain a partially
  hyperbolic NU2SE attractor with \emph{three-dimensional
    center-unstable bundle} $d_{cu}=3$ and hyperbolic equilibria which
  are non-sectional hyperbolic; and
\item in Subsection~\ref{sec:section-hyperb-equil}: we adapt the
  previous example so that the equilibria are again of mixed-type:
  generalized Lorenz-like and generalized Rovella-like.
\item in Section~\ref{sec:3-section-expans}: we build partially
  hyperbolic attractors which are asymptotic $p$-sectional-hyperbolic
  and \emph{not non-uniformly $(p-1)$-sectional expanding}, for any
  $p>2$.
\end{itemize}

\begin{remark}[strong ASH]
  \label{rmk:strongASH}
  The examples obtained in Section~\ref{sec:multid-exampl} satisfy
  the wNU2SE condition for all trajectories of the ambient manifold
  not converging to the singularities -- which is a stronger form of
  wASE condition.
\end{remark}

\begin{remark}[asymptotic $p$-sectional expansion]
  \label{rmk:apASH}
  The examples obtained in Section~\ref{sec:3-section-expans} satisfy
  the following for $p>2$
  \begin{description}
  \item[weak non-uniform $p$-sectional expansion (wNUpSE)] there
    exists $c_0>0$ so that
\begin{align}\label{eq:NUpSE}
\liminf_{n\nearrow\infty}\frac1n\sum\nolimits_{i=0}^{n-1}
  \log
  \big\|\wedge^p (Df\mid_{E^{cu}_{f^i x}})^{-1}\big\|
  \le - c_0
\end{align}
for all points $x$ in $M$ whose trajectories do not converge to
singularities.
\end{description}

A similar proof to~\cite[Theorem 1.6]{ArSal25} shows that the above
wNUpSE condition implies asymptotic $p$-sectional hyperbolicity (pASH)
as defined in Subsection~\ref{sec:asympt-p-section}.
\end{remark}

\subsection{Comments and conjectures}
\label{sec:comments-conjectures}

Sets satisfying sectional-hyperbolicity or asymptotic sectional
hyperbolicity also satisfy the Hyperbolic
Lemma~\ref{le:hyplemma}. This desirable property does not extend to
$p$-sectional hyperbolic attractors, as the Shilnikov-Turaev wild
attractor example~\cite{ST98} shows.

As we note in Remarks~\ref{rmk:nu2se-ase} and~\ref{rmk:apASH}, it is
known that wNU2SE condition implies wASE for all trajectories not
converging to a singularity, as well as the corresponding
$p$-sectional version.

\begin{conjecture}\label{conj:ase-nu2se}
  A partially hyperbolic compact invariant set for a smooth flow
  satisfying the ASE condition (pASE) also satisfies the NU2SE (NUpSE)
  condition.
\end{conjecture}

Theorems~\ref{thm:fABV} and~\ref{thm:physASH}, together with recent
results from Burguet-Ovadia~\cite{BOvadia}, show that the existence of
physical/SRB measures should only depend on positive Lyapunov
exponents and not on slow recurrence.

\begin{conjecture}[physical/SRB without slow recurrence]
  \label{conj:physnoSR}
 In Theorem~\ref{thm:discretefabv} and Theorem~\ref{mthm:hdASH} the
 slow recurrence conditions are superfluous.
\end{conjecture}

Naturally, we should study the existence of physical measures for
$p$-sectional expanding partially hyperbolic attractors.

\begin{conjecture}[$p$-sectional expansion and physical measure]
  \label{conj:physpsec}
  The attractors constructed in Section~\ref{sec:3-section-expans}
  admit a unique physical/SRB measure with a full volume ergodic basin.
\end{conjecture}

\subsection{Acknowledgments}
\label{sec:acknowledgements}

V.A. thanks the Mathematics and Statistics Institute of the Federal
University of Bahia (Brazil) for its support of basic research and
CNPq (Brazil --- grant 304047/2023-6) for partial financial
support. L.S. thanks the Mathematics Institute of Universidade Federal
do Rio de Janeiro (Brazil) for its encouraging of mathematical
research; and the Mathematics and Statistics Institute of the Federal
University of Bahia (Brazil) for its hospitality, together with
CAPES-Finance code 001, CNPq, FAPERJ - Aux{\'\i}lio B\'asico \`a
Pesquisa (APQ1) Project E-26/211.690/2021; and FAPERJ - Jovem
Cientista do Nosso Estado (JCNE) grant E-26/200.271/2023 for partial
financial support; and also PROEXT-PG project Dynamic Women -
Din\^amicas, CNPq-Brazil (grant Projeto Universal 404943/2023-3).

\section{Proof of existence of a physical/SRB measure}
\label{sec:existence-physic-mea-1}

Here we prove Theorem~\ref{mthm:hdASH}. The proof relies on applying
the following useful extension of Pesin's Formula~\cite{Pe77} obtained
by Catsigeras-Cerminara-Enrich~\cite{CatCerEnr15}.

\begin{theorem}[Generalized Pesin's Inequality~\cite{CatCerEnr15}]
  \label{thm:GenPesin}
  For any $C^1$ diffeomorphism $f$, if $\Lambda$ is an invariant
  compact set with a dominated splitting $T_\Lambda M= E\oplus F$,
  then for Lebesgue almost every point $x$ satisfying
  $\omega(x)\subset \Lambda$, the entropy of any weak$^*$ limit
  measure $\mu$ of the sequence
  $\big(\frac1n \sum_{i=0}^{n-1} \delta_{f^i(x)}\big)_{n\ge1}$ is
  bounded from below:
  \begin{align}\label{eq:gpesin}
    h_\mu(f)\geq \int \log|\det (Df\mid_{F})|\,d\mu.
  \end{align}
\end{theorem}

We used this in~\cite{ArSal25} for $x$ satisfying the (wNU2SE)
condition to find an ergodic hyperbolic physical/SRB measure as an
ergodic component of a limit measure $\mu$ as above. It is well known
from the work of Ledrappier-Young~\cite{LY85} that for $C^2$ systems
such measures are physical/SRB measures.

In what follows we write $J^{cu}f(w):=|\det(Df\mid_{E^{cu}_w})|$ and
$\psi^{cu}(w):= \log\big\|\wedge^2 (Df\mid_{E^{cu}_w})^{-1}\big\|$.

\begin{proof}[Proof of Theorem~\ref{mthm:hdASH}]
  We start by using Theorem~\ref{thm:GenPesin} to fix a full
  $\leb$-measure subset $X\subset U$ and for $x\in X$ we consider a
  weak$^*$ limit point $\mu$ of the sequence considered in Theorem
  ~\ref{thm:GenPesin}. Then we have
  $h_\mu(f)\ge\int\log J^{cu}f\,d\mu$ from~\eqref{eq:gpesin}.

  We want to obtain the reverse inequality to conclude Pesin's Formula
  $h_\mu=\mu(\log J^{cu}f)>0$ and invoke Ledrappier-Young's main
  result from~\cite{LY85} ensuring, since $f$ is a $C^2$
  diffeomorphism, that $\mu$ admits an ergodic component $\nu$ which
  is a physical/SRB measure, as needed.

  We consider the following two cases: either $\mu$ admits some atom ---
  necessarily a periodic point of $f$ --- or $\mu$ is non-atomic.

  In the latter case, then $\mu(\sing_\Lambda(G))=0$ and a
  $\mu$-generic point $z\in\Lambda$ cannot converge to any
  singularity. Indeed, otherwise we would have for any continuous
  observable $\vfi:M\to\RR$
  \begin{align*}
    \mu(\vfi)=\int\vfi\,d\mu
    =
    \lim_{n\to+\infty}\frac1n\sum\nolimits_{i=0}^{n-1}\vfi(f^iz)
  \end{align*}
  since $z$ is Birkhoff-generic for $\mu$, and
  $\mu(\vfi)=\vfi(\sigma)$ if $z\in W^s(\sigma)$ for some singularity
  $\sigma\in\sing_\Lambda(G)$. Because this holds for any continuous
  observable, we conclude $\mu=\delta_\sigma$ contradicting the
  assumption that $\mu$ is non-atomic.

  Therefore, such Birkhoff-generic $z$ must be in
  $\Lambda^*:=\Lambda\setminus
  \cup\{W^s_\sigma:\sigma\in\sing_\Lambda(G)\}$ and so, by assumption,
  such $z$ satisfies wNU2SE. In particular, using $\psi^{cu}$ as
  observable together with Birkhoff Ergodic Theorem, we obtain
  $\mu(\psi^{cu})\le -c_0$.

  This ensures, in particular, that all Lyapunov exponents along
  $E^{cu}$ are either zero (along the flow direction $G$) or strictly
  positive. Hence, we get that
  $\mu(\log J^{cu}f)=\int\Sigma^+\,d\mu\ge c_0$ gives the averaged sum
  of central-unstable Lyapunov exponents, and we can apply Ruelle's
  Inequality
  \begin{align*}
    h_\mu(f)\le\int\Sigma^+\,d\mu = \mu(\log J^{cu}f).
  \end{align*}
  We conclude that $\mu$ satisfies  Pesin's Formula as needed, and the
  existence of an ergodic component $\nu$ of $\mu$ which is a
  physical/SRB measure follows.

  We are left with the case where $\mu$ admits an atomic
  component. Let $\mu=\alpha\cdot\eta+\beta\cdot\xi$ with
  $\alpha+\beta=1, \alpha\ge0, \beta>0$ be the decomposition of $\mu$
  into a non-atomic component $\eta$ and a purely atomic component
  $\xi$.  We want to use our condition (A) or (B) to show that this
  case cannot occur.

  Atomic components of invariant probability measures are at most
  denumerably many; such components for a continuous transformation
  are supported on periodic orbits.  Periodic orbits of the time-$1$
  map $f$ of the flow of $G$ are either equilibria $\sigma$ (fixed
  points of the flow) or periodic orbits $p$ of the flow with integer
  minimal period.

  If $p\in\supp\xi$ has minimal period $\ell$, then
  $\pi_p:=\frac1\ell\sum_{i=0}^{\ell-1}\delta_{f^ip}$ is an ergodic
  component of $\xi$ and
  $\pi_p(\log J^{cu}f)=\frac1{\ell}\log J^{cu}f^\ell(p)\ge c_*>0$
  since $p\in\Lambda^*$.

  We start with condition (A). Then for any
  $\sigma\in\sing_\Lambda(G)\cap \supp\xi$ we have
  $\delta_\sigma(\log J^{cu}f)>0$ and we cannot have $\beta=1$,
  i.e., $\mu$ cannot be purely atomic. Indeed, in this case
  $h_\mu(f)=0\ge\mu(\log J^{cu}f)=\xi(\log J^{cu}f)$ and we can write
  \begin{align}\label{eq:xilogJ}
    \xi(\log J^{cu}f)
    =
    \left(\sum\nolimits_{\sigma}t_\sigma\delta_\sigma
    +
    \sum\nolimits_{i\ge1}t_i\pi_{p_i}
    \right)(\log J^{cu})>0, 
  \end{align}
  where $\sum_{i\ge1}t_i+\sum_\sigma t_\sigma =1$ and
  $t_i,t_\sigma\ge0$. This contradicts the previous inequality.

  This contradiction ensures that there exists a non-atomic component
  with positive mass: that is, $\alpha>0$. Since $\eta=\mu-\xi$ is
  also $f$-invariant, we can write
  \begin{align*}
    h_{\mu}(f)
    &=
    \alpha h_{\eta}(f)+\beta h_{\xi}(f)
    =
    \alpha \cdot h_{\eta}(f)
    \\
    &\ge
    \mu(\log J^{cu}f)
    =
    \alpha \eta(\log J^{cu}f)+\beta \xi(\log J^{cu}f)
    \ge
    \alpha \cdot \eta(\log J^{cu}f).
  \end{align*}
  We have obtained an $f$-invariant non-atomic probability measure $\eta$
  satisfying $h_{\eta}(f)\ge \eta(\log J^{cu}f)$. We can now proceed
  with the same argument as before to obtain a physical/SRB measure.

  We now replace condition (A) with condition (B). Then equilibria
  cannot be atoms of $\mu$ by Remark~\ref{rmk:noatom} However,
  from~\eqref{eq:xilogJ} other periodic points $p$ are also excluded
  since $p\in\Lambda^*$. Therefore, we conclude that $\beta=0$ in this
  case and we recover that $\mu$ is non-atomic. The rest of the
  argument follows and the proof of existence of a physical/SRB
  measure is complete.

  Finally, if $\Lambda$ is transitive, then the ergodic basin
  $B(\mu)$ of $\mu$ covers a neighborhood of $\Lambda$, except perhaps a
  zero volume subset, as a consequence of the properties of $cu$-Gibbs
  states; see e.g.~\cite[Theorem 5.1]{ArSal25}.
\end{proof}

\section{Construction of  ASH attractors}
\label{sec:multid-exampl}

We consider a ``solenoid'' constructed over a uniformly expanding map
$g:\TT\to\TT$ of the $k$-dimensional torus $\TT$, for some
$k\ge2$. That is, let $\DD$ be the unit disk on $\RR^2$ and consider a
smooth embedding $F_0:N\circlearrowleft$ of $N=\TT\times\DD$ into
itself, which preserves and contracts the foliation
$\F^s=\big\{\{z\}\times\DD: z\in\TT\big\}$. We will write $E^s$ for
the tangent bundle to the leaves of this foliation. The natural
projection $\pi:N\to\TT$ on the first factor \emph{smoothly
  conjugates} $F_0$ to $g$: $\pi\circ F_0=g\circ\pi$ --- we can assume
that $\pi$ is the projection associated to a tubular neighborhood of
$F_0(N\times\{0\})$. We assume also that
$F_0$ admits a fixed point $p$ and that
$\lambda_1^{-1}\le\|(Dg)^{-1}\|\le\lambda_0^{-1}$ for some fixed
$\lambda_1>\lambda_0>1$.

\subsection{ASH attractor, with non sectionally
  hyperbolic equilibria}
\label{sec:multid-asympt-sectio}

We start with $k=1$, that is, with the three-dimensional Smale
solenoid map.

\subsubsection{The suspension of the solenoid map}
\label{sec:suspens-soleno-map}

We further consider the constant vector field $X:=(0,1)$ on
$M_0=N\times[0,1]$, which defines a transition map from
$\Sigma_\epsilon=N\times\{\epsilon\}$ to
$\Sigma_{1-\epsilon}=N\times\{1-\epsilon\}$ for some fixed small
$\epsilon>0$, which is the identity in the first coordinate when
restricted to $\Sigma_\epsilon$. Next we modify this field on the
cylinder $\cC=U\times\DD\times[0,1]$ around the periodic orbit of the
point $p=(z,0)\in N\times\{0\}$, where $U$ is a small neighborhood of
$z$ in $\TT$, in such a way as to create both two equilibria
$\sigma_1,\sigma_2$, with either $k$ expanding and $3$ contracting
eigenvalues, or $1$ expanding and $k+2$ contracting eigenvalues, as
follows.  We fix $k=1$, so that $\TT=\sS^1$ and $U$ is an interval;
and ignore the stable foliation along $\DD$ in the next arguments.

\begin{figure}[htbp]
\includegraphics[width=10cm,height=9cm]{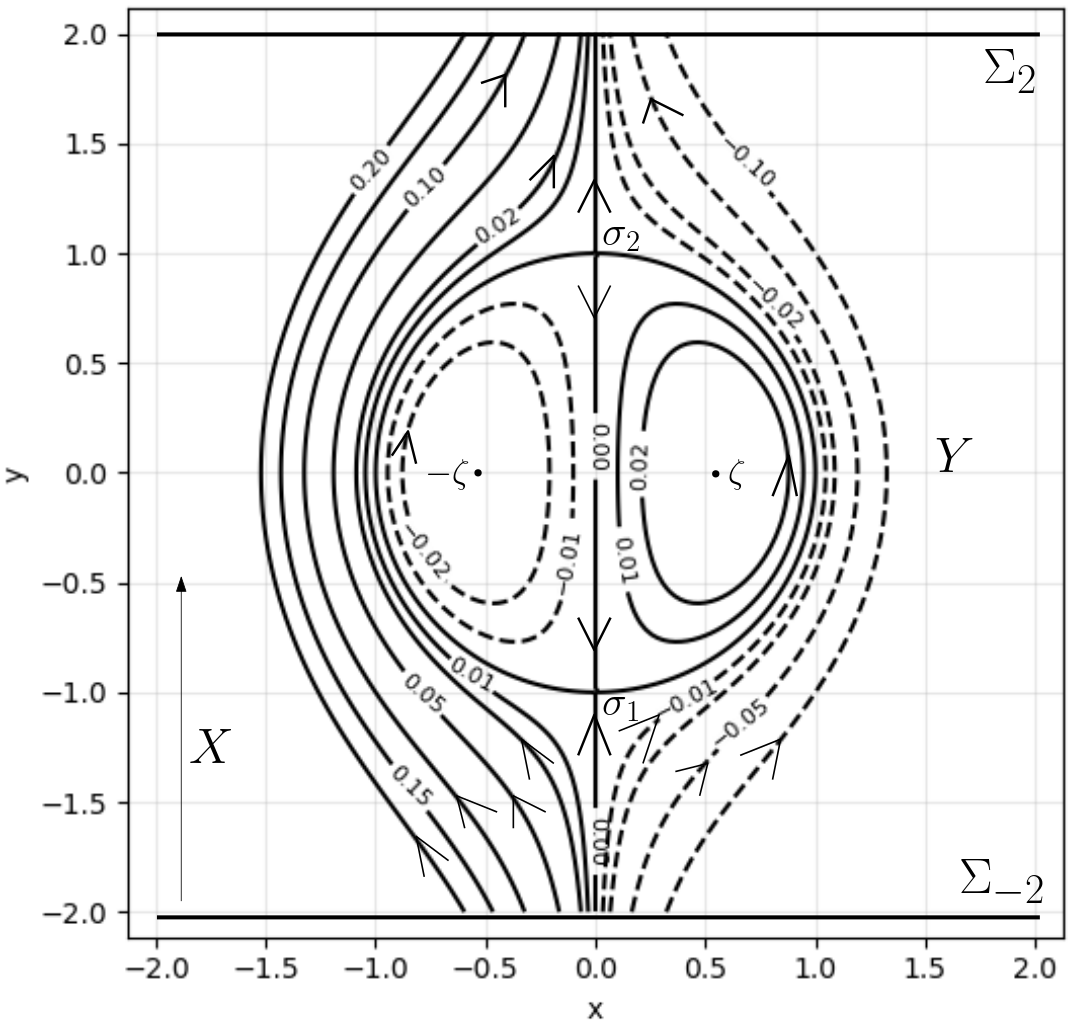}
\caption{A sketch of the vector field $Y$ with its Poincar\'e map from
  $\Sigma_{-2}$ to $\Sigma_2$ compared with the vector field $X$ in
  the square $[-2,2]\times[-2,2]$.}
\label{fig:H0}
\end{figure}
  
We consider the vector field $Y_0$ on the rectangle
$C_0=[-3,3]\times[-2,2]$ depicted in Figure~\ref{fig:H0} obtained from
the level curves of
$H(x,y):=x\big(1-(x^2+y^2)\xi_0(x)-2(1-\xi(x))\big)/10$, where
$\xi:\RR\to[0,1]$ is a smooth bump function $\xi_0:\RR\to[0,1]$ so
that $\xi_0\mid_{[-2,2]}\equiv1$,
$\xi_0\mid_{\RR\setminus[-3,3]}\equiv0$, $\xi_0(-x)=\xi_0(x)$ and
$\xi_0\mid_{\RR^+}$ is decreasing.

The field $Y_0$ is the  Hamiltonian vector field
$Y_0:= (H^\prime_y, -H^\prime_x)$ in the plane.
This ensures that \emph{$Y_0$ is conservative restricted to $C_0$} ---
recall that we are ignoring the uniformly contracting directions $E^s$
tangent to the foliation $\cF^s$.

We note that for $x\ge3$ we have $Y=X$.  Moreover, for $x\le2$ we can
explicitly write
 \begin{align*}
   Y_0(x,y)
   &=
     \left(-\frac{x y}5 , \frac{3x^2+y^2-1}{10}\right).
 \end{align*}

\subsubsection{Poincar\'e transition map is the identity}
\label{sec:poincare-transit-map}

The symmetry of the Hamiltonian ensures that the level curves of $H$
with non-zero values are symmetric with respect to the transformation
$S:(x,y)\mapsto(-x,y)$, i.e., $S(H^{-1}(\{\zeta\}))=H^{-1}(\{-\zeta\})$ for
$\zeta\neq0$ and these level curves are trajectories of the flow with
positive speed in the $y$ direction. 

\begin{claim}\label{cl:trId}
  The transition Poincar\'e map from $\Sigma_{-2}=[-3,3]\times\{-2\}$
  to $\Sigma_2=[-3,3]\times\{2\}$ is the identity away from the point
  $(0,-2)$
\end{claim}

\begin{remark}[time to cross]
  \label{rmk:crosstime}
  For $2\le|x|\le3$ the flow on $C_0$ has a vertical speed along the
  positive direction of the $y$-axis of at least
  $(2^2-1)/10=3/10$. Hence, starting from $(x,-2)$ the flow arrives at
  $(x,2)$ after a time of $t(x)\le 4\cdot 10/3\le16$.
\end{remark}

\subsubsection{Non-sectional hyperbolic equilibria}
\label{sec:non-section-hyperb}
  
The eigenspace of one of the contracting (expanding) eigenvalues of the
equilibria $\sigma_1,\sigma_2$ lies along the vertical direction (the
direction of $X$), \emph{the other two-dimensional contracting
  directions still lie on the direction of $\DD$ (ignored in the
  pictures)}, and the remaining expanding/contracting eigenspaces are
transversal to the vertical $X$ direction; see
Figure~\ref{fig:H0}. There are also a pair of fixed elliptic
equilibria represented by $\pm\zeta$.

We have $\sigma_i=(0,(-1)^i), i=1,2$ and $\zeta=(\sqrt3/3,0)$
so that
\begin{align}\label{eq:DY0}
  DY_0(\sigma_i)=
  \begin{bmatrix}
    (-1)^{i+1}/5
    & 0
    \\ 0
    & (-1)^i/5
  \end{bmatrix}
      \quad\&\quad
      DY_0(\pm\zeta)
      =\pm
      \begin{bmatrix}
        0
        & -\sqrt3/15
        \\
        \sqrt3/25
        &0
      \end{bmatrix}.
\end{align}
This shows that $\sigma_i$ are hyperbolic saddles which are \emph{not
  sectionally hyperbolic}: neither sectionally expanding, nor
sectionaly contracting, since their traces vanish.

\subsubsection{Partial hyperbolic attractor}
\label{sec:partialy-hyperb}

After rescaling, we assume that $Y_0$ is defined in the initial cylinder
$\cC$, by setting the coordinates corresponding to the factor $\DD$
equal to zero.  We also assume that the standard inner product
satisfies $\langle Y_0, X\rangle>0$ on the Poincar\'e sections
$\Sigma_{\epsilon}\cup\Sigma_{1-\epsilon}$ corresponding to
$\Sigma_{-2}\cup\Sigma_2$; and take a $C^\infty$ bump function
$\psi:[0,1]\circlearrowleft$ so that
$\psi\mid_{[\epsilon/2,1-\epsilon/2]}\equiv0$ and
$\psi\mid_{[0,\epsilon/3]\cup[1-\epsilon/3,1]}\equiv1$. Then, we
define the vector field
\begin{align}\label{eq:attached}
  G_0(x,u):=
  \psi(u)\cdot X+(1-\psi(u))\cdot Y_0(x,u), \quad (x,u)\in M_0
\end{align}
which generates a smooth transition map $L$ from
$\Sigma_0^*=(N\setminus\{p\})\times\{0\}$ to
$\Sigma_1=N\times\{1\}$. \emph{Since the Poincar\'e transition maps of
  both $X$ and $Y$ are the identity, then $L=Id$.}

Together with the identification $(x,0)\sim(F_0(x),1), x\in N$ we
obtain a smooth parallelizable manifold $M=M_0/\sim$ where $G_0$ induces
a $C^\infty$ vector field which we denote by the same letter. We write
$(\phi_t)_{t\in\RR}$ for the induced flow.
We also have an attracting subset $\Lambda=\cap_{t\ge0}\phi_t(M)$ with
$M$ as topological basin of attraction.

\emph{The Poincar\'e first return map of this vector field
  $P:\Sigma_0^*\to\Sigma_0$ coincides with
  $F_0\mid_{N\setminus\{p\}}$.} In particular,
$\Lambda_0:= \cap_{n\in\ZZ_0^+}F_0^n(N)$ has an open and dense subset
of dense trajectories, and so the flow $\phi_t$ of $G$ is transitive
on $\Lambda$. Thus, $\Lambda$ is an attractor.

Since $\Lambda_0$ admits a $DF_0$-invariant hyperbolic splitting
$T_{\Lambda_0}N=E^s_{\Lambda_0}\oplus E^u_{\Lambda_0}$, and we may
assume without loss of generality that the contracting rate along
$E^s$ is stronger than the contracting eigenvalues of
$\sigma_1, \sigma_2$, then setting
\begin{align*}
  E^s_{(w,t)}:=D\phi_t(E^s_w) \quad \& \quad
  E^{cu}_{(w,t)}:=D\phi_t(E^u_w)\oplus \RR\cdot
  X, \quad w\in\Lambda_0, t\in[0,1);
\end{align*}
we obtain a $D\phi_t$-invariant and continuous splitting
$T_\Lambda M = E^s \oplus E^{cu}$ which is partially hyperbolic.

\subsubsection{Asymptotical sectional expansion}
\label{sec:asympt-section-expan}

Since the area along any $2$-plane of $T(U\times[0,1])$ is preserved,
we get $\psi^{cu}(w)=\log\|\wedge^2 (D\phi_1(w)\mid E^c_w)^{-1}\|=0$.

Hence, given $w\in\Sigma_0^*$ whose future trajectory visits $\cC$
infinitely many times, there exist sequences $n_i<m_i<n_{i+1}$ of
iterates so that $n_0=0$ and, for $j=n_i,\dots,m_i-1$, we have
$f^jw=\phi_1^j(w)\in\cC$; and $f^jw\in M\setminus\cC$ for
$j=m_i,\dots,n_i-1$. The previous argument shows that
\begin{align}\label{eq:cancel}
  \sum\nolimits_{j=n_i}^{m_i-1}\psi^{cu}(f^jw)=0.
\end{align}
Since on $M\setminus\cC$ the time-$1$ map on $\Sigma_0^*$ coincides
with $P$, then for $f^i(w)\in M\setminus\cC$ we can assume that
$f^i(w)\in\Sigma_0^*$ and obtain
\begin{align}\label{eq:NUSE00}
  \sum\nolimits_{j=m_i}^{n_i-1}\psi^{cu}(f^jw)
  \le
  -(n_i-m_i)\log\lambda_0.
\end{align}
Thus, we can write (grossly underestimating the number of iterates
outside of $\cC$ from $0$ to $n_{i+1}$ by $i$)
\begin{align*}
  \frac1{n_{i+1}}\sum\nolimits_{j=0}^{n_{i+1}-1}\psi^{cu}(f^jw)
  \le
  \frac{i-n_i}{n_{i+1}}\log\lambda_0. 
\end{align*}
We note that close to the stable manifold $W^s(\sigma_1)$ of
$\sigma_1$ in $\cC$ the time-$1$ map takes a potentially unbounded
amount of iterates to cross $\cC$ from bottom to top. Moreover, given
$\epsilon>0$ we can find $\delta>0$ so that any $w\in\Sigma_0^*$,
which visits a small neighborhood $B_\delta(p)$ away from $p$
infinitely many times, satisfies $i/n_i\le\epsilon$ as
$i\to\infty$. Thus
$\frac{i-n_i}{n_{i+1}}\log\lambda_0<(1-\epsilon_0)\log\lambda_0$ for
all large enough $i$.

This shows that the flow satisfies the (wNU2SE) condition for
all trajectories which visit $B_\delta(p)\setminus\{p\}$ infinitely
many times or just a finite number of times. All trajectories of
$w\in\Lambda$ which do not converge to $\sigma_1$
(i.e. $w\in\Lambda\setminus W^s(\sigma_1)$) as well as all points of
the stable leaf $W^s(w)$ in $M$ satisfy this.

Since points $w$ whose trajectories do not pass through the point $p$
form a full Lebesgue measure subset, together with
Remark~\ref{rmk:nu2se-ase}, we have shown that the attractor $\Lambda$
satisfies both (wNU2SE) and (wASE).

\begin{remark}\label{rmk:wNU2SElarge}
  Moreover, this also holds for all trajectories of the ambient space
  $M$ which do not converge to the singularities.
\end{remark}
Hence, the flow admits a unique physical/SRB measure $\mu$ from
Theorem~\ref{thm:physASH} with full basin: $\leb(M\setminus B(\mu))=0$.

\begin{remark}[non-robustness]
  The properties of the flow $Y_0$ are clearly not robust: the
  perturbation $\wh{Y}(x,y):=(H'_y(x,y)-x/10, -H'_x(x,y))$ has
  trajectories sketched in Figure~\ref{fig:HamFlowPert}.
  \begin{figure}[htpb]
    \centering
    \includegraphics[width=7cm,height=6cm]{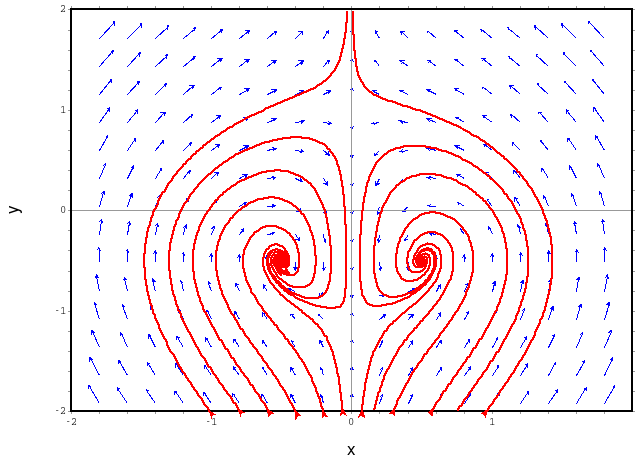}
    \caption{\label{fig:HamFlowPert} Sketch of the flow of the vector
      field $\wh{Y}$ with some trajectories, showing that trajectories
      starting close to $(0,-2)$ (that is, close to $p$ in the
      original suspension flow) will fall into sinks.}
  \end{figure}
\end{remark}

\subsection{ASH attractor with mixed sectionally
  hyperbolic equilibria}
\label{sec:multid-ash-attract}

We now modify the previous example to get sectional-hyperbolic
equilibria.  We keep the following symmetry relations from
$Y_0(x,y)=\big(Y_0^1(x,y),Y_0^2(x,y)\big)$
\begin{align}\label{eq:symmetry}
  Y_0^1(-x,y)=-Y_0^1(x,y)=Y_0^1(x,-y)
  \quad\&\quad
  Y_0^2(\pm x,\pm y)=Y_0^2(x,y)
\end{align}
by setting
$$Y_1(x,y) := \big(H^\prime_y(x,y), -2\cdot H^\prime_x(x,y)\big).$$
It is straightforward to check that
\begin{itemize}
\item $\diver (Y_1)=-H^{\prime\prime}_{yx}=y/5$; 
\item the $y$-axis $\{x=0\}$ is still invariant; and
\item the points $\sigma_1,\sigma_2$ and $\pm\zeta$ are still
  equilibria;
\item for the equlibria $\sigma_1, \sigma_2$ we obtain the same
  properties as in~\eqref{eq:DY0} with the second row multiplied by
$2$.
\end{itemize}
This provides sectional hyperbolicity: $\sigma_1$ becomes a
sectionally contracting (``Rovella-like'') singularity along
$E^{cu}_{\sigma_1}$; while $\sigma_2$ becomes a sectionally expanding
(``Lorenz-like'') singularity along $E^{cu}_{\sigma_1}$.

\begin{remark}[crossing time]
  \label{rmk:crosstime2}
  Again, for $2\le|x|\le3$ the flow $Y_1$ on $C_0$ has a vertical speed
  along the positive direction of the $y$-axis of at least
  $(2^2-1)/5=3/5$. Hence, starting from $(x,-2)$ the flow arrives at
  $(x,2)$ after a time of $t(x)\le4\cdot 5/3\le8$.
\end{remark}



\subsubsection{Poincar\'e transition map is the identity}
\label{sec:poincare-transit-map-1}

This vector field also satisfies Claim~\ref{cl:trId} since we have the
following properties of the trajectories of $Y_1$. 
\begin{lemma}[symmetric solutions]
  \label{le:symmetric}
  For any $\epsilon>0$, the trajectories
  $\big(\gamma(t)\big)_{t\in(-t_0,t_0)}$ of a vector field $Y_1$ with
  $\gamma(t)=\big(x(t),y(t)\big)$ satisfying~\eqref{eq:symmetry} are
  such that
  \begin{enumerate}[(a)]
  \item $\wt\gamma(t):=-\gamma(t), t\in(-t_0,t_0)$ is a
    trajectory of $-Y_1$; and
  \item $\wh{\gamma}(t):=(x(t),-y(t)), t\in(-t_0,t_0)$ is also a
    trajectory of $-Y_1$.
  \end{enumerate}  
\end{lemma}

\begin{proof}
  Just observe that since $x'(t)=Y_1^1\big(\gamma(t)\big)$ and
  $y'(t)=Y_1^2\big(\gamma(t)\big)$ 
  \begin{align*}
    \wt\gamma^\prime(t)
    &=
      -\gamma'(t)
      =
      -Y_1\big(\gamma(t)\big)
      =
      -Y_1\big(-\gamma(t)\big)
    =
    -Y_1\big(\wt\gamma(t)\big); \qand
    \\
    \wh{\gamma}^\prime(t)
    &=
      \big(x'(t),-y'(t)\big)
      =
      \Big(Y_1^1\big(x(t),y(t)\big) ,
      -Y_1^2\big(x(t),y(t)\big)\Big)
    \\
    &=
      \Big( - Y_1^1\big(x(t),-y(t)\big) ,
      -Y_1^2\big(x(t),-y(t)\big)\Big)
      =
      -Y_1(\wh\gamma(t));
  \end{align*}
  for each $-t_0<t<t_0$.
\end{proof}

To prove the claim, we note that from Lemma~\ref{le:symmetric}, for
each trajectory $\gamma(t)$ starting at $(x_0,-2)$ with $x_0\neq0$ and
crossing $\cC$ to a point $(x_1,2)$, there corresponds a trajectory
$\wh\gamma(t)$ of $-Y_1$, which starts at $(x_0,2)$ and crosses to the
points $(x_1,-2)$; see Figure~\ref{fig:symmetry}.

\begin{figure}
  \includegraphics[height=5cm]{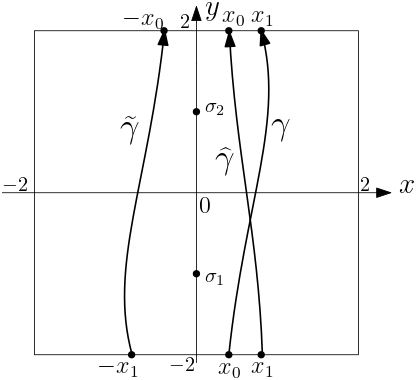}
  \caption{\label{fig:symmetry} Sketch of the trajectories $\gamma$,
    $\wt\gamma$ and $\wh\gamma$ of $Y_1$ if $0<x_0<x_1$.}
\end{figure}

We claim that $x_0=x_1$.  Arguing by contradiction, if $0<x_0<x_1$, we
have a pair of trajectories of a flow starting at $(x_0,-2)$ and
$(x_1,-2)$; and crossing to the points $(x_1,2)$ and $(x_0,2)$. Since
the order was exchanged, there must be an intersection of the
trajectories. This contradiction proves the claim.

Property (a) from Lemma~\ref{le:symmetric} ensures that the same
argument holds for $x_0<0$.
Since all trajectories starting at $(x_0,-2)$ with $x_0\neq0$ cross to
$(x_0,2)$, we have proved Claim~\ref{cl:trId}.

\begin{claim}[heteroclinic connection]
  We have the heteroclinic connection $W^u(\sigma_1)=W^s(\sigma_2)$.
\end{claim}

Indeed, since all points $(x_0,-2)$ with $x_0\neq0$ cross to a point
$(x_0,2)$, then we keep the heteroclinic connection --- for otherwise,
some trajectory in $W^u(\sigma_1)$ would cross to a point $(x_1,2)$
for some $x_1\neq 0$ (since the $y$-axis
is still invariant) and this contradicts the symmetry of solutions.

\subsubsection{Asymptotically sectional expansion}
\label{sec:asympt-section-expan-1}

Considering the flow $G_1$ obtained from $Y_1$ as
in~\eqref{eq:attached}, the symmetry of the flow on $\cC$ enables us
to use similar arguments as in
Subsection~\ref{sec:asympt-section-expan}: given $w\in\Sigma_0^*$
whose trajectory crosses $\cC$ infinitely many times, we consider the
sequence $n_i<m_i<n_{i+1}$ of iterates bounding visits of the
trajectory of $w$ in $\cC$; and note that the regions
\begin{align*}
  \{\wh{w} \in \cC: \psi^{cu}(\wh{w})=\diver(Y_1)(\wh{w})>0\}
  \qand
  \{\wh{w} \in \cC: \psi^{cu}(\wh{w})=\diver(Y_1)(\wh{w})<0\}
\end{align*}
are symmetric with the same size: they correspond to the upper half
($y>0$) and lower half ($y<0$) of the cylinder $C_0$.  Both are
traversed by each trajectory through $\cC$ in a symmetric way using
the same number of iterates modulo a finite difference. Hence, we can
write~\eqref{eq:cancel}. We also keep~\eqref{eq:NUSE00}.  So (wNU2SE)
and (wASH) follow. In addition, Remark~\ref{rmk:wNU2SElarge} also
holds.

From Lemma~\ref{le:symmetric} we have that the attracting set admits
dense trajectories, and so we have an wASH attractor with an unique
physical measure from Theorem~\ref{thm:physASH}. 

\subsection{Higher co-dimensional ASH attractor}
\label{sec:multidimesional-ash}

Now we restart with $k=\ell+1$ the previous suspension flow
construction, for any fixed $\ell\ge1$, and adapt the two-dimensional
vector field $Y_1$ by considering the vector field $Y_2$ in the
cylinder $C_1=[-3,3]\times B^\ell\times[-2,2]$, where $\|y\|_2$ is the
Euclidean norm; $B^\ell=\{y\in\RR^\ell:\|y\|_2\le 3\}$ is the closed
unit ball with radius $3$ centered at the origin, and $Y_2$ is given
by
\begin{align}\label{eq:Y2}
  Y_2(x,y,z):=(H'_z(x,z),\omega\cdot\xi_1(\|y\|_2^2)\cdot y,-2H'_x(x,z)),
  \quad (x,y,z)\in C_1
\end{align}
for some fixed $\omega>0$. Here, $\xi_1:\RR\to[0,1]$ is a smooth bump
function such that $\xi_1\mid_{\RR^+}$ is decreasing;
$\xi_1(-t)=\xi_0(t)$ for all $t\in\RR$; $\xi_1\mid_{[-4,4]}\equiv1$
and $\xi_1\mid_{\RR\setminus[-9,9]}\equiv0$.

\begin{remark}[explicit solution and crossing time]
  \label{rmk:crosstime3}
  We can explicitly solve for $y$ with initial condition
  $y_0\in B^\ell$ with $\|y\|_2<2$: $y(t)=y_0e^{\omega t}$ for
  $0\le t\le \omega^{-1}\log(2/\|y_0\|_2)$.

  Since the $x,z$-components of $Y_2$ coincide with the components of
  $Y_1$, the crossing time of the flow from $\Sigma_{-2}$ to
  $\Sigma_2$ is again at most $8$ for $2\le |x|\le3$.
\end{remark}

\subsubsection{(Sectional-)Hyperbolicity of equilibria}
\label{sec:section-hyperb-equil-1}

It is easy to see that the equilibria are
$\sigma_i=(0,0,(-1)^i), i=1,2$ and $\zeta=(\sqrt3/3,0,0)$; where
$\zeta$ is non-hyperbolic, $\sigma_{1,2}$ are both hyperbolic of
saddle-type, with $\sigma_2$ generalized Lorenz-like and $\sigma_1$
generalized Rovella-like.

Moreover, we have for $w=\sigma_i, i=1,2$
\begin{align}
  \label{eq:cuJac}
  \psi^{cu}(w)
  =
    (-1)^{i+1}/5 \qand
    \log J^{cu}f(w)
  =
    \omega + (-1)^{i+1}/5.
\end{align}

\subsubsection{Sectional hyperbolicity along $C_1$}
\label{sec:section-hyperb-along}

We can write, using the product structure of $C_1$, for $w\in C_1$,
$w=(x,y,z)$ so that $\varrho^2=x^2+\|y\|_2^2\le4$
\begin{align*}
  DY_2(x,y,z)=
  \begin{pmatrix}
    -z/5 & 0 & -x/5
    \\
    0 & \omega & 0
    \\
    6x/5 & 0 & 2z/5
  \end{pmatrix};
\end{align*}
and so the subbundles $E^1_w:=\RR\times 0^\ell\times\RR$ and
$E^2_w=0\times\RR^\ell\times 0$ are $D\phi_t$-invariant under the flow
$\phi_t$ of $Y_2$ inside the subcylinder $\varrho\le4$. While $w$ and
$\phi_{[0,t]}(w)$ are in this subcylinder, for some $t>0$, we have the
following domination property
\begin{align*}
  \|D\phi_t\mid E^1_w\|
  \le
  e^t
  <
  e^{\omega t }
  =
  \|\big(D\phi_t\mid E^2_w)^{-1}\|;
\end{align*}
as long as $\omega>1$.
This ensures that the least expansion along any $2$-subspace
by $D\phi_1\mid E^c_w$ at $w=(x,y,z)\in C_1$ is achieved along the
$E^1_w$-subbundle, that is
\begin{align*}
  \big\|\wedge^2\big( D\phi_1\mid E^c_w \big)^{-1}\big\|
  =
  \big\|\wedge^2\big( D\phi_1\mid E^1_w \big)^{-1}\big\|
  =
  \big\|\wedge^2\big(D\wh{\phi}_1\mid E^c_{(x,z)}\big)^{-1}\big\|;
\end{align*}
where $(\wh{\phi}_t)_{t\in\RR}$ is the flow of the vector field $Y_1$
from the previous subsection.

The symmetry of the trajectories in the $(x,z)$ variables together
with the above choice of $\omega$, ensures that, for each trajectory
crossing $C_1$, the portion of the trajectory covering the region
$z<0$ contributes to the sum~\eqref{eq:cancel} by the same amount, but
of opposite sign, as the portion of the same trajectory covering the
region $z>0$.  Hence, we reobtain~\eqref{eq:cancel}.

  \subsubsection{Asymptotical sectional expansion}
  \label{sec:asympt-section-expan-6}
  
  We observe that we do not necessarily have the Poincar\'e map $P$
  coinciding with the original expanding map $g$, since now we do not
  have symmetry on the $y$ variable, although the transition map from
  $(x,y,-2)$ to $(x,\bar y,2)$ keeps the $x$-variable.

  By construction, the transition map expands the $y$-variable close
  to $0$, but admits a contraction away from $0$ due to the use of the
  bump function in the definition of $Y_2$. Nevertheless, by
  Remark~\ref{rmk:crosstime3} the crossing time of the possible
  contracting region is at most $2$, the contraction rate is bounded;
  and we may assume the value of $\lambda_0>1$ large enough so that
  the transition map of the flow of the vector field $G_3$, obtained
  from $Y_3$ by the same procedure as~\eqref{eq:attached}, is still
  uniformly expanding.  Therefore, we also keep~\eqref{eq:NUSE00} with
  $\wh{\lambda_0}>1$ in the place of $\lambda_0$.

  Thus, the same argument from the previous
  Subsections~\ref{sec:asympt-section-expan}
  and~\ref{sec:asympt-section-expan-1} provides the wNU2SE property,
  and consequently wASH after Remark~\ref{rmk:nu2se-ase}, on all
  trajectories not converging to a singularity, as in
  Remark~\ref{rmk:wNU2SElarge}.
  
  Thus, $\Lambda$ becomes an wASH attractor with $d_{cu}=\ell+2$ for
  any fixed positive integer $\ell \ge1$.

  Finally, we from~\eqref{eq:cuJac} we have condition (A) of
  Theorem~\ref{mthm:hdASH}, and so we can ensure existence and
  uniqueness of a physical/SRB measure for this flow. 

\section{Construction of non-uniformly sectional expanding
   attractors}
\label{sec:new-examples-non}

We extend symmetrically the vector fields $Y_0, Y_1$ to
three-dimensional versions first.

\subsection{Higher co-dimensional NU2SE with non-sectional hyperbolic
  equilibria}
\label{sec:multid-ash-}

We now assume that $k=2$ and rotate the setup of Figure~\ref{fig:H0}
around the vertical axis: we set for
$(\varrho,\theta,z)\in[0,3]\times[0,2\pi]\times[-2,2]$
\begin{align*}
  Y_3(\varrho \cos\theta, \varrho \sin\theta, z) :=
  H^\prime_y(\varrho,z)\cdot(\cos\theta,\sin\theta,0)-
  H^\prime_x(\varrho,z)\cdot X;
\end{align*}
the corresponding symmetrized vector field from the plane hamiltonian
vector field $Y_0$.

\begin{remark}[consequences of symmetry]
  \label{rmk:symmetry}
  The symmetry ensures that $Y_3$ also satisfies
  Claim~\ref{cl:trId}. Hence, the flow $\phi_t$ of $G_3$ induces a
  transition map $L$ from $\Sigma_0^*=(N\setminus\{p\})\times\{0\}$ to
  $\Sigma_1=N\times\{1\}$ \emph{which is the identity}, as in
  Subsection~\ref{sec:poincare-transit-map}.

  Moreover, every vertical plane containing the $z$-axis, i.e, with
  equation $ax+by=0$ for any pair $(a,b)\neq(0,0)$, is preserved by
  the flow along with its area.
\end{remark}

\subsubsection{Hyperbolic and non-sectional hyperbolic equilibria}
\label{sec:hyperb-non-section}

We can write more explicitly for $\varrho\le2$, since
$\varrho^2=x^2+y^2$ 
\begin{align*}
  Y_3(x,y,z)
  &=
    -\frac{\varrho z}5\cdot\frac{(x,y,0)}{\sqrt{x^2+y^2}}
    +\left(0,0,\frac{3\varrho^2+z^2-1}5\right)
    =
    \left(-\frac{xz}5,-\frac{yz}5,\frac{3x^2+3y^2+z^2-1}{10}\right)
\end{align*}
and it is now easy to calculate
\begin{align*}
  DY_3(x,y,z)
  =
  \begin{pmatrix}
    -z/5 & 0 & -x/5
    \\
    0 & -z/5 & -y/5
    \\
    3x/5 & 3y/5 & z/5
  \end{pmatrix}
                  \qand
                  \diver(Y_3)\equiv -z/5.
\end{align*}
It is easy to see that equilibria are given by the pair
$\sigma_i=(0,0,(-1)^i), i=1,2$, of hyperbolic saddles which are not
sectionally hyperbolic; and
$\zeta(\alpha):=(\sqrt3/3)\cdot(\cos\alpha,\sin\alpha,0),
\alpha\in[0,2\pi)$, a circle of elliptical fixed points.

We attach $Y_3$ to the suspension flow $X$ as in~\eqref{eq:attached}
obtaining a $(k+3)$-dimensional flow $G_3$ (recall that here $k=2$ and
the stable direction is two-dimensional).



\subsubsection{Asymptotic sectional expansion}
\label{sec:asympt-section-expan-3}

From Remark~\ref{rmk:symmetry} at any point $w\in C_0$ there exists a
$2$-plane whose area is preserved by $D\phi_t(w)=Df(w)$.

For a point $w\in(U\setminus\{p\})\times\{0\}$ close to $p$, the time
$\tau(w)$ needed to cross $\cC$ can be estimated as
$\tau(w)\le C\cdot|\log d(w,p)|$ for some constant $C>0$, since $p$
belongs to the stable manifold of the hyperbolic saddle equilibria
$\sigma_1$. This ensures that $\tau$ is $\leb$-integrable.

The exterior product $\|\wedge^2\big(D\phi_t\mid E^c_w\big)^{-1}\|$ is
bounded above by $\|D\phi_t\|^2$, and from the Linear Variational
Equation and the Gronwall's Inequality $\|D\phi_t(w_0)\|\le e^{t \|DY_3\|}$,
where $\|DY_3\|=\sup_{w\in \cC}\|DG_3(w)\|$ for any $w_0\in\cC$. Hence
we get
\begin{align}\label{eq:highest}
  \|DY_3\|\tau(w)
  \ge
  \sum\nolimits_{i=0}^{[\tau(w)]} \psi^{cu}(f^i(w)).
\end{align}
Arguing as in Subsection~\ref{sec:asympt-section-expan}, given
$w\in\Sigma_0^*$ whose future trajectory visits $\cC$ infinitely many
times, we consider the same sequences $n_i<m_i<n_{i+1}$ of iterates
marking the crossings of $\cC$. From the above arguments, we
keep~\eqref{eq:NUSE00} and use~\eqref{eq:highest} to
replace~\eqref{eq:cancel} by the following
\begin{align}
  \label{eq:cancel1}
  \sum\nolimits_{j=n_i}^{m_i-1}\psi^{cu}(f^i(w))
  \le
  \|DY_3\|\cdot\tau\big(f^{n_i}w\big).
\end{align}
Now we can estimate
\begin{align}
  \sum\nolimits_{j=0}^{n_{i+1}-1}
  &\psi^{cu}(f^i(w))
  \le
  -\log\lambda_0\sum\nolimits_{k=0}^i(n_{k+1}-m_k)
    +\|DY_3\|\sum\nolimits_{k=0}^i(m_k-n_k)\nonumber
  \\
  &=\label{eq:estima0}
    -\log\lambda_0\cdot\#\{0\le j<n_{i+1}: \phi_1^j(w)\in
    \Sigma_0^*\setminus U\}
    +\|DY_3\|\sum\nolimits_{k=0}^i\tau\big(f^{n_k}w\big).
\end{align}
Since on $M\setminus\cC$ the time-$1$ map on $\Sigma_0^*$ coincides
with $P$, then we can recount the iterates of $\phi_1^j(w)$ through
the iterates $P^k(w)$: we set $\tau\mid_{\Sigma_0^*\setminus U}\equiv1$ and
$\ell(i):=i+\sum_{k=0}^i (n_{k+1}-m_k)$ the lap number, note that
$n_{i+1} = \sum_{k=0}^{\ell(i)}\tau(P^kw)$ and
rewrite~\eqref{eq:estima0} as
\begin{align*}
  -\log\lambda_0
  &\cdot\#\{0\le k<\ell(i): P^k(w)\in
    \Sigma_0^*\setminus U\}
  \\
  &+\|DY_3\|\sum\{\tau\big(P^kw\big): 0\le k<\ell(i): P^k(w)\in U\}
  \\
  &=
    \sum\nolimits_{j=0}^{\ell(i)}
    \Big(
    \big(-\log\lambda_0\mathbf{1}_{\Sigma_0^*\setminus U}
    +\|DY_3\|\mathbf{1}_U \big)\cdot
    \tau\Big)\circ P^j (w).
\end{align*}
Hence, for $w\in (U\setminus\{p\})\times\{0\}$ we can write
\begin{align}
  \frac1{n_{i+1}}
  &\sum_{j=0}^{n_{i+1}-1} \psi^{cu}(f^jw)
    \le \nonumber
    -\frac{\log\lambda_0}{n_{i+1}}
    \sum_{k=0}^{\ell(i)}\big(\mathbf{1}_{\Sigma^*_0\setminus U}\big)(P^kw)
    +
    \frac{\|DY_3\|}{n_{i+1}}
    \sum_{k=0}^{\ell(i)}\big(\tau\mathbf{1}_{U}\big)(P^kw)
    \\
  &=\label{eq:medias}
    -\frac{\ell(i)\log\lambda_0}{n_{i+1}}
    \frac1{\ell(i)}\sum_{k=0}^{\ell(i)}
    \big(\mathbf{1}_{\Sigma^*_0\setminus U}\big)(P^kw)
    +
    \frac{\ell(i)\|DY_3\|}{n_{i+1}}\frac1{\ell(i)}
    \sum_{k=0}^{\ell(i)}\big(\tau\mathbf{1}_{U}\big)(P^kw),
\end{align}
and by ergodicity of $\leb_\Sigma$ with respect to $P$, since
$\tau\mid_{\Sigma_0^*\setminus U}\equiv1$ and
\begin{align}\label{eq:intau}
  \frac{n_{i+1}}{\ell(i)}
  =
  \frac1{\ell(i)}\sum\nolimits_{k=0}^{\ell(i)}\tau(P^kw)
  \xrightarrow[i\to+\infty]{}\leb(\tau)=\int\tau\,d\leb_{\Sigma},
  \quad \leb_\Sigma-\text{a.e. } w\in\Sigma_0^*;
\end{align}
we arrive at
\begin{align*}
  \limsup_{i\to\infty}\frac1{n_{i+1}}
  \sum\nolimits_{j=0}^{n_{i+1}-1}\psi^{cu}(f^jw)
  \le
  -\frac{\log(\lambda_0)}{\leb(\tau)}
  (1-\leb_\Sigma(U))
  +
  \frac{\|DY_3\|}{\leb(\tau)}\leb_\Sigma(\tau\mathbf{1}_{U}).
\end{align*}
Finally, since $\tau\in L^1(\leb_\Sigma)$ and $U=B_\epsilon(p)$, we
can make $\leb_\Sigma(\tau\mathbf{1}_{U})=\int_U \tau\,d\leb_\Sigma$
as close to zero as needed.

Since trajectories eventually returning to a full
$\leb_\Sigma$-measure subset of $\Sigma_0^*$ form a full volume subset
of the ambient manifold $M$, we can thus conclude NU2SE as long as $U$
is small enough.

\subsubsection{Slow recurrence}
\label{sec:slow-recurr-equilibr}

To obtain a physical/SRB measure with full basin it is enough to
obtain slow recurrence according to Theorem~\ref{thm:discretefabv}. We
explore the invariance and ergodicity of $\leb_\Sigma$ with respect to
the Poincar\'e first return map $P$, the integrability of the
Poincar\'e first return time, together with the symmetry of the flow
on the cylinder $\cC$.

We use the equivalence between the SR condition~\eqref{eq:SR} and its
continuous version: on a positive volume subset of points, for every
$\epsilon>0$, we can find $r>0$ so that
\begin{align}
  \label{eq:SSR}
  \limsup_{T\nearrow\infty}\frac1T\int_0^T-\log d_r
  \big(\phi_t(x),\sing_\Lambda(G)\big) \, dt <\epsilon;
\end{align}
see~\cite[Theorem C]{ArSal25} for the proof of the stated equivalence.

In what follows we write
$\Delta_r(x):=-\log d_r \big(\phi_t(x),\{\sigma_1,\sigma_2\}\big)$ and
consider trajectories starting at a point $x\in\Sigma_0^*$ on a subset
with full $\leb_\Sigma$-measure, and claim that we can find a constant
$C>0$ such that for all small $r>0$
\begin{align}
  \label{eq:Lebergodic}
  \limsup_{T\nearrow\infty}\frac1T\int_0^T\Delta_r\big(\phi_t(x)\big) \, dt
  \le
  C\int_{\|u\|_2<r}\big((\log \|u\|_2)^2-(\log r)^2\big) d\lambda_2(u);
\end{align}
where $\lambda_2$ is the Lebesgue measure on
$\RR^2$. 
It is easy to see that the above expression tends to zero when
$r\to0+$, as we need.

\subsubsection*{Reduction to plane dynamics}

From Remark~\ref{rmk:symmetry}, each trajectory crossing $\cC$ is
contained in one vertical plane through the $z$-axis. We can assume,
without loss of generality, that we are dealing with a flow like $Y_0$,
whose trajectories are depicted in Figure~\ref{fig:H0}, to
estimate the value of the integral in~\eqref{eq:SSR}.

Considering $0<r\ll 1$, then trajectories outside of $\cC$ do not
contribute to the above integral --- we consider only those entering
$\cC$ through a small neighborhood $I_0=(-r,r)\times\{-2\}$ on
$\Sigma_{-2}$. We assume (without loss of generality) that from $I_0$
to the ball $B_r(\sigma_1)$ the flow is essentially tubular: starting
at $(x_0,-2)$ we will arrive at $B_r(\sigma_1)$ with the first
coordinate still equal to $x_0$. Likewise, between $B_r(\sigma_1)$ and
$B_r(\sigma_2)$ and from $B_r(\sigma_2)$ and $(x_0,2)$ we assume that
the flow is tubular.

From~\cite[Theorem 1.3]{Newhouse16} we can locally $C^1$ linearize the
flow around $\sigma_1$ in the ball $B_r(\sigma_1)$ (reducing the value
of $r>0$ if needed): there exists a $C^1$ diffeomorphism
$\zeta:B_r(\sigma_1)\to\RR^2$ so that
$\zeta(\phi_t(w))=e^{Dt}\zeta(w)$ for $w\in B_r(\sigma_1), t>0$ so
that $\phi_{[0,t]}(w)\subset B_r(0)$ and $D=\diag\{1/5,-1/5\}$.

Therefore, the distance $d(\phi_t(w),\sigma_1)$ can be estimated by
$\|e^{tD}\zeta(w)\|_2$ in the Euclidean norm, and so the integral
in~\eqref{eq:SSR} for a trajectory starting at the boundary of
$B_r(\sigma_1)$ can be calculated, writing $\zeta(w)=(x_0,r)$ with
$x_0\neq0$
\begin{align*}
  \int_0^t-\frac12\log\|e^{tD}\zeta(w)\|_2^2\,dt
  &=
  -\frac12\int_0^t\log(e^{2s/5}x_0^2+e^{-2s/5}r^2)\,ds
  \le
  -\frac12\int_0^t\log(e^{2s/5}x_0^2)\,ds
  \\
  &=
  -\int_0^t(s/5+\log |x_0|)\,ds
  =
  -t(t/10+\log |x_0|).
\end{align*}
The trajectory leaves $B_r(\sigma_1)$ before the time $t_0$ so that
$e^{t_0/5}|x_0|=r \iff t_0=5\log(r/|x_0|)$. Thus each trajectory crossing
$B_r(\sigma_1)$ contributes to the integral in~\eqref{eq:SSR} by at
most
$S=-t_0(t_0/10+\log |x_0|)=\frac52\big( (\log |x_0|)^2-(\log r)^2\big)$.

The second coordinate of $\zeta(\phi_{t_0}(w))$ at the exit from
$B_r(\sigma_1)$ is $e^{-t_0/5}r=x_0$ again. From $B_r(\sigma_1)$ to
$B_r(\sigma_2)$ we can likewise assume that the flow is tubular, and
repeat the calculation again when crossing $B_r(\sigma_2)$.

We thus obtain that at each crossing of $\cC$ starting at $(x_0,-2)$
we arrive at $(x_0,2)$ after a time $\tau(x_0,-2)$ and
\begin{align}
  \label{eq:S1}
  \int_0^{\tau(x_0,-2)}\Delta_r\big(\phi_t(w)\big)\,dt
  \le
  C\cdot (2 S) = 5 C \big( (\log |x_0|)^2-(\log r)^2\big)
\end{align}
for some constant $C>0$.

\subsubsection*{Back to the dynamics on $M$}

We now consider a trajectory starting at $\leb$-generic point
$w\in\Sigma_0^*$ and crossing $\cC$ through $I_0$ at times
$t_n<T_n<t_{n+1}$ so that $P^{k_n}w=\phi_{t_n}w\in I_0$ and
$T_n=t_n+\tau(P^{k_n}w)$, where $P^{k_i}w$ are precisely those
iterates which fall in $I_0$.  At every visit to $B_r(z)\times{0}$ on
$N\times\{0\}$ the expression $|x_0|$ in the upper bound
from~\eqref{eq:S1} means $d(P^{k_n}w,z)$.  We can estimate as follows
\begin{align*}
  \int_0^{T_n}\Delta_r\big(\phi_s(w)\big)\,dt
  &=
    \sum\nolimits_{i=1}^n\int_{t_i}^{T_i}
    \Delta_r\big(\phi_s(w)\big)\,dt
  \\
  &\le
  5C\sum\nolimits_{i=1}^{k_n}
    \big((\log d(P^iw,z))^2-(\log r)^2\big)
    \cdot\mathbf{1}_{B_r(z)}(P^iw).
\end{align*}
Thus, we can estimate the average as
\begin{align*}
  \frac1{T_n}\int_0^{T_n}\Delta_r\big(\phi_s(w)\big)\,dt
  &\le
    \frac{5C\cdot k_n}{T_n}\cdot\frac1{k_n}
    \sum_{i=1}^{k_n}
    \big((\log d(P^iw,z))^2-(\log r)^2\big)
    \cdot\mathbf{1}_{B_r(z)}(P^iw);
\end{align*}
but we also have $T_n=\sum_{i=0}^{k_n}\tau(P^iw)$ (recall the
definition of $\tau$ as the Poincar\'e time associated to the
Poincar\'e first return map) so that we
can use  the $P$-invariance and ergodicity of $\leb_\Sigma$ to get
\begin{align*}
  \limsup_{n\to\infty}
  \frac1{T_n}\int_0^{T_n}\Delta_r\big(\phi_s(w)\big)\,dt
  &\le
    \frac{5C}{\leb_\Sigma(\tau)}\int_{B_r(z)}
    \big( (\log d(w,z))^2-(\log r)^2\big)\,d\leb_\Sigma(w)
  \\
  &=
    \frac{5C}{\leb_\Sigma(\tau)}
    \int_{u\in B_r(0)\subset\RR^2}\big((\log \|u\|_2)^2-(\log
    r)^2\big)\,d\lambda_2(u),
\end{align*}
where $\lambda_2$ is the Lebesgue area measure on the Euclidean plane.

Given any strictly increasing and unbounded positive real sequence
$(s_m)_{m\ge1}$, we have the following two cases
\begin{description}
\item[$T_{n_m}<s_m<t_{n_m+1}$] we get the bound
  \begin{align*}
    \frac1{s_m}\int_0^{s_m}\Delta_r\big(\phi_t(w)\big)\,dt
    =
    \left(\frac{T_{n_m}}{s_m}\right)\cdot\frac1{T_{n_m}}
    \int_0^{T_{n_m}}\Delta_r\big(\phi_t(w)\big)\,dt
    \le
    \frac1{T_{n_m}}
    \int_0^{T_{n_m}}\Delta_r\big(\phi_t(w)\big)\,dt; 
  \end{align*}
\item[$t_{n_m}\le s_n<T_{n_m}$] we get the bound
    \begin{align*}
      \frac1{s_m}\int_0^{s_m}\Delta_r\big(\phi_t(w)\big)\,dt
    \le
    \left(\frac{T_{n_m}}{s_m}\right)\cdot\frac1{T_{n_m}}
    \int_0^{T_{n_m}}\Delta_r\big(\phi_t(w)\big)\,dt.
    \end{align*}
  \end{description}
Since $T_n=t_n+\tau(P^{k_m}w)>s_m\ge t_m$ then
\begin{align}\label{eq:quotient}
  \frac{T_{n_m}}{s_m}
  \le
  \frac{t_{n_m}+\tau(P^{k_m}w)}{t_{n_m}}
  =
  1+\frac{\tau(P^{k_m}w)}{t_{n_m}}
  =
  1+\frac{\tau(P^{k_m}w)/k_m}{\frac1{k_m}\sum_{i=0}^{k_m-1}\tau(P^iw)}.
\end{align}
For $\leb_\Sigma$-a.e. $w$, from $P$-invariance and ergodicity we have
\begin{align*}
  \frac1{k_m}\sum_{i=0}^{k_m-1}\tau(P^iw) \to \leb_\Sigma(\tau)
  \qand
  \frac{\tau(P^{k_m}w)}{k_m}\to 0;
\end{align*}
so that~\eqref{eq:quotient} tends to $1$ for large $m$.
Altogether, this shows that
\begin{align*}
  \limsup_{n\to\infty}
  \frac1{s_m}\int_0^{s_m}\Delta_r\big(\phi_t(w)\big)\,dt
  =
\limsup_{n\to\infty}
  \frac1{T_n}\int_0^{T_n}\Delta_r\big(\phi_s(w)\big)\,dt;
\end{align*}
completing the proof of~\eqref{eq:Lebergodic}.

\subsection{Higher co-dimensional NU2SE with mixed
  sectional-hyperbolic equilibria}
\label{sec:section-hyperb-equil}

We repeat the construction starting with the vector field $Y_2$
from~Subsection~\ref{sec:multid-ash-attract}, that is, we consider
\begin{align*}
  Y_4(\varrho \cos\theta, \varrho \sin\theta, z)
  :=
  H^\prime_y(\varrho,z)\cdot(\cos\theta,\sin\theta,0)-
  2\cdot H^\prime_x(\varrho,z)\cdot X.
\end{align*}
We note that the action of the flow $\phi_t$ of $Y_4$ on $2$-planes is
given by the additive compound $\wedge^{[2]}DY_4$: i.e. given that
$D\phi_t(w)$ is the solution of the Linear Variational Equation on
$\RR^3$
\begin{align*}
  Z'=DY_4(\phi_t(w))\cdot Z, \qquad Z_0=I_3:\RR^3\to\RR^3;
\end{align*}
then $\wedge^2D\phi_t(w)$ is the solution of
\begin{align*}
Z'=\wedge^{[2]}DY_4(\phi_t(w))\cdot Z, \qquad Z_0=I_3:\RR^3\to\RR^3;
\end{align*}
and we can use the following
\begin{align*}
  A =
  \begin{pmatrix} a_{11} & a_{12} & a_{13} \\ a_{21} & a_{22} &
  a_{23} \\ a_{31} & a_{32} & a_{33} \end{pmatrix}
\implies
\wedge^{[2]}A = \begin{pmatrix}
a_{11} + a_{22} & a_{23} & -a_{13} \\
a_{32} & a_{11} + a_{33} & a_{12} \\
-a_{31} & a_{21} & a_{22} + a_{33}
\end{pmatrix}                  
\end{align*}
(see e.g.~\cite{Muldowney1990} or~\cite{fiedler74,Zhang2013} for short
introductions, where $\wedge^{[2]}A$ is written $A^{[2]}$)
so that for $\varrho\le2$ we get $\diver(Y_4)\equiv 0$ and
\begin{align*}
  DY_4(x,y,z)
  =
  \begin{pmatrix}
    -z/5 & 0 & -x/5
    \\
    0 & -z/5 & -y/5
    \\
    6x/5&6y/5& 2z/5
  \end{pmatrix}                 
               \,\&\,
               \wedge^{[2]}DY_4(x,y,z)=
  \begin{pmatrix}
    -2z/5 & -y/5 & x/5
    \\
    6y/5 & z/5 &0
    \\
    -6x/5 &0& z/5
  \end{pmatrix}.
\end{align*}
Therefore, the hyperbolic equilibra $\sigma_{1,2}$ became
sectional-hyperbolic
\begin{align*}
  DY_4(0,0,\pm1)
  =
  \pm
  \begin{pmatrix}
    -1/5 & 0 & 0
    \\
    0 & -1/5 & 0
    \\
    0 & 0  & 2/5
  \end{pmatrix}                 
               \,\&\,
               \wedge^{[2]}DY_4(0,0,\pm1)=\pm
  \begin{pmatrix}
    -2/5 & 0&0
    \\
    0 & 1/5 &0
    \\
    0 &0& 1/5
  \end{pmatrix};
\end{align*}
with $\sigma_1$ generalized Rovella-like and $\sigma_2$ generalized
Lorenz-like.

\subsubsection{Asymptotical sectional-expansion}
\label{sec:asympt-section-expan-4}

From the arguments of Subsection~\ref{sec:poincare-transit-map-1},
since $Y_4$ is based on a symmetrization of $Y_2$, we recover
Claim~\ref{cl:trId} so that the Poincar\'e transition map on
$\Sigma_0^*$ is again the identity.

The upper bound from~\eqref{eq:highest} is kept with
$\|DY_4\|= \sup_{w\in \cC}\|DG_4(w)\|$ in the place of $\|DY_3\|$; and
so the same argument from Subsection~\ref{sec:asympt-section-expan-3}
shows that the flow of $G_4$ --- obtained from $X$ by attaching $Y_4$
as in~\eqref{eq:attached} --- satisfies NU2SE as long as $U$ is small
enough.

\subsubsection{Slow recurrence to equilibria}
\label{sec:slow-recurrence-}

The same argument of~\ref{sec:slow-recurr-equilibr} applies here, with
similar upper bound, to conclude the SR condition on a full volume
subset of $M$. Again, from Theorem~\ref{thm:discretefabv} we conclude
that there exists a unique physical/SRB measure whose basin covers
$\leb$-a.e. point of the ambient space.

\section{Example of $p$-sectional expansion without asymptotic
  $(p-1)$-sectional-expansion}
\label{sec:3-section-expans}

We repeat the construction on Subsections~\ref{sec:multid-ash-}
and~\ref{sec:section-hyperb-equil} ensuring that trajectories spend a
large enough time in the cylinder $\cC$. For this we attach either
$Y_i$  to the
original laminar flow as
\begin{align}\label{eq:attached1}
  \wh{G_i}(w,u):=
  \psi(u)\cdot X+(1-\psi(u))\zeta(u)\cdot Y_i(w,u), \quad (w,u)\in
  M_0, \quad i=3,4;
\end{align}
where $\psi$ is the same as in~\eqref{eq:attached}, and
$\zeta:\RR\to[1-\zeta_0,1]$ is a $C^\infty$ function so that
\begin{align*}
  \zeta\mid_{(-\infty,\epsilon]\cup[1-\epsilon,+\infty)}\equiv1 \qand
  \zeta\mid_{[2\epsilon,1-2\epsilon]}\equiv1-\zeta_0
\end{align*}
for some fixed $0\le\zeta_0<1$, reducing the speed of the vector field
inside the cylinder.

As in the previous examples, \emph{the Poincar\'e first return map of
  this vector field $P:\Sigma_0^*\to\Sigma_0$ coincides with
  $F_0\mid_{N\setminus\{p\}}$.}

\subsection{Asymptotic $3$-sectional expansion}
\label{sec:asympt-3-section}

The same symmetry and frequency arguments from
Subsection~\ref{sec:asympt-section-expan-1} shows that the vector
field $\wh{G_i}$ obtained from $Y_i$ following the attaching
procedure~\eqref{eq:attached1} is $3$-sectionally expanding, since
$\diver(Y_i)$ corresponds to the rate of change of volume along the
$3$-dimensional fiber $E^c$, for $i=3,4$.

Since the argument is valid for all trajectories in the attractor not
converging to equilibria, we obtain an \emph{asymptotic $3$-sectional
  hyperbolic attractor}.

\subsection{Absense of asymptotical ($2$-)sectional expansion}
\label{sec:asympt-section-expan-5}

\subsubsection{With non-sectional hyperbolic equilibria}
\label{sec:with-non-sectional}

For $Y_3$ we can calculate for $\varrho\le 2$, and assuming
$\zeta_0=0$ for simplicity
\begin{align}\label{eq:2additive3}
  \wedge^{[2]}DY_3(x,y,z)=
  \begin{pmatrix}
    -2z/5 & -y/5 & x/5
    \\
    3y/5 &0&0
    \\
    -3x/5 &0&0 
  \end{pmatrix}.
\end{align}
This operator has eigenvalues
\begin{align}\label{eq:eigenvalues3}
  \lambda_2=(-z-\sqrt{z^2-3\varrho^2})/5; \quad
  \lambda_1=(-z+\sqrt{z^2-3\varrho^2})/5
  \qand \lambda_0=0;
\end{align}
and respective eigenvectors 
\begin{align}\label{eq:eigenvectors3} 
  v_2
  =
  (\varrho^2, 5y\lambda_1, -5x\lambda_1);
  \quad
  v_1
  =
  (\varrho^2, 5y\lambda_2, -5x\lambda_2)
  \qand
  v_0=(0,y,x).
\end{align}
From the expressions~\eqref{eq:2additive3},~\eqref{eq:eigenvalues3}
and~\eqref{eq:eigenvectors3}, we have the following cases for $Y_3$
(recall that $\varrho^2=x^2+y^2$).
\begin{description}
\item[For $z>0$] if $z^2\ge 3\varrho^2$ we have
$\lambda_2\le -z/5 \le \lambda_1<0$; otherwise we get
$\lambda_2=\bar\lambda_1\in\CC\setminus\RR$ with
$\Re(\lambda_2)=\Re(\lambda_1)=-z/5<0$.
\end{description}
Thus, at the region
$z>0, \varrho\le2$ of $C$, \emph{there always is a $2$-plane which is
  contracted by $D\phi_t$} at a rate $\le -z/5$. Hence we can estimate
\begin{align}\label{eq:derlog}
  \sigma(x,y,z):=\frac{\partial}{\partial t}
  \left.
  \log\|\wedge^2\big(D\phi_t\mid E^c_{(x,y,z)}\big)^{-1}\|
  \right|_{t=0}
  \ge
  \frac{z}5, \quad z>0.
\end{align}
\begin{description}
\item[For $z\le0$] from the preservation of area along the plane
  orthogonal do $v_0$, we obtain
\begin{align}\label{eq:derlog0}
  \sigma(x,y,z)
  \ge
  0, \quad z\le0.
\end{align}
\end{description}

\subsubsection{With mixed type sectional-hyperbolic equilibria}
\label{sec:with-mixed-type}

For $Y_4$, from Subsection~\ref{sec:section-hyperb-equil}, we
calculate the eigenvalues of $\wedge^{[2]}DY_4(x,y,z)$ as follows
\begin{align*}
  \lambda_2=(-z-\sqrt{9z^2-24\varrho^2})/10; \quad
  \lambda_1=(-z+\sqrt{9z^2-24\varrho^2})/10
  \qand \lambda_0=z/5;
\end{align*}
with the respective eigenvectors
\begin{align*}
    v_2
  =
  (2\varrho^2, y(10\lambda_1-2z), x(3z-10\lambda_1);
  \quad
  v_1
  =
  (2\varrho^2, y(10\lambda_2-2z), x( 2z-10\lambda_2))
\end{align*}
and $v_0=(0,y,x)$. Analogously, we split in two cases.
\begin{description}
\item[For $z>0$] if $z^2\ge (8/3) \varrho^2$ we have
$\lambda_2\le -z/10 < \lambda_1<2z$; otherwise we get
$\lambda_2=\bar\lambda_1\in\CC\setminus\RR$ with
$\Re(\lambda_2)=\Re(\lambda_1)=-z/10<0$.
\end{description}
 Thus, at the region
$z>0, \varrho\le2$ of $C$, \emph{there always is a $2$-plane which is
  contracted by $D\phi_t$} at a rate $\le -z/10$. Hence we can estimate
\begin{align}\label{eq:derlog1}
  \varsigma(x,y,z):=\frac{\partial}{\partial t}
  \left. \psi^{cu}(x,y,z)
  \right|_{t=0}
  \ge
  \frac{z}{10}, \quad z>0.
\end{align}
\begin{description}
\item[For $z\le0$] since $\lambda_0=z/5$, we again have a contracted
  $2$-plane at a rate $z/5$; thus
\begin{align}\label{eq:derlog2}
  \varsigma(x,y,z)\ge-\frac{z}5, \quad z\le0.
\end{align}
\end{description}

\subsubsection{Lower bound for sectional-expansion}
\label{sec:lower-bound-section}

Finally,
from~\eqref{eq:derlog},~\eqref{eq:derlog0},~\eqref{eq:derlog1}
and~\eqref{eq:derlog2} we have
$\log\|\wedge^2\big(D\phi_t\mid E^c_w\big)^{-1}\| \ge \int_0^t
\varsigma(\phi_s(w))\,ds$ and so for the vector field $Z$ equal to
either $Y_3$ or $Y_4$ we can write
\begin{align}
  \|DZ\|\tau(w)
  \ge
  \sum\nolimits_{i=0}^{[\tau(w)]}
  &\psi^{cu}(f^iw)
  \ge \nonumber
  \log\|\wedge^2\big(D\phi_{\tau(w)}\mid E^c_w\big)^{-1}\|
  \\
  &\ge \label{eq:lowest}
    \int_0^{[\tau(w)]} \varsigma(\phi_s(w))\,ds
    \ge
    \int_{[\tau(w)/2]}^{[\tau(w)]}\frac{\pi_3(\phi_s(w))}5 \,ds
    = a(w)\tau(w);
\end{align}
where $a(w)>0$ and $\pi_3$ is the projection on the third coodinate in
$C_0$; and we have used the symmetry of the flow inside the cylinder
(so that $\varsigma(\phi_s(w))>0 \iff s>\tau(w)/2$ for $Y_3$ and we use a
loose lower bound for $Y_4$).

We can use~\eqref{eq:lowest} to replace~\eqref{eq:cancel} by the
following
\begin{align}
  \label{eq:cancel2}
  \|DZ\|\cdot\tau\big(f^{n_i}w\big)
  \ge
  \sum\nolimits_{j=n_i}^{m_i-1}\psi^{cu}(f^jw)
  \ge
  a\big(f^{n_i}w\big)\cdot\tau\big(f^{n_i}w\big).
\end{align}
We can also replace~\eqref{eq:NUSE00} by
\begin{align}\label{eq:NUSE1}
  \sum\nolimits_{j=m_i}^{n_i-1} \psi^{cu}(f^jw)
  \ge
  -(n_i-m_i)\log\lambda_1.
\end{align}

\subsubsection{Slowing the flow on $\cC$}
\label{sec:slowing-flow-cc}

Choosing $0<\zeta_0<1$ in the definition of $Z$, we change the
bound~\eqref{eq:lowest} on the rate of sectional
expansion/contraction as follows.

Since $f:[\tau(w)/2,\tau(w)]\to[0,2], t\mapsto \pi_3(\phi_t(w))/5$ is
smooth, bijective and strictly monotonous, we can use its inverse
$h:[0,2]\to[\tau(w)/2,\tau(w)]$ to change variables
\begin{align*}
  \int_{h(0)}^{h(2)}f
  =
  \int_0^2 (f\circ h)\cdot h'
  =
  \int_0^2\frac{s\,ds}{f'(h(s))}
  =
  \int_0^2\frac{ds}{\|Y_\ell(\phi_{h(s)}(w))\|};
\end{align*}
where we used that
\begin{align*}
  f'(t)
  =
  \frac15 D\pi_3(\phi_t(w))\cdot \partial_t(\phi_t(w))
  =
  \frac15\pi_3(\phi_t(w))\cdot\|Y_\ell(\phi_t(w))\|
  =
  f(t)\cdot\|Y_\ell(\phi_t(w))\|.
\end{align*}
Hence, the value of $a(w)$, obtained for $\zeta_0=0$, is multiplied by
$(1-\zeta_0)^{-1}$ when we pass to some $0<\zeta_0<1$.

Thus, in~\eqref{eq:cancel2}, increasing the value of $\zeta_0\in(0,1)$
not only increases the value of $\tau(w)$ --- due to reduced speed of
the flow on $\cC$ --- but introduces the factor $(1-\zeta_0)^{-1}$ in
the lower bound.

\subsubsection{Using the integrability of $\tau$ and ergodicity of $P$}
\label{sec:using-ergodicity-p}

Given $w\in\Sigma_0^*$ close to $p$ whose future trajectory visits
$\cC$ infinitely many times, there exist sequences $n_i<m_i<n_{i+1}$
of iterates so that $n_0=0$ and, for $j=n_i,\dots,m_i-1$, we have
$\phi_1^j(w)\in\cC$; and $\phi_1^j(w)\in M\setminus\cC$ for
$j=m_i,\dots,n_i-1$.

We can now estimate using the previous lower bounds, similarly to
Subsection~\ref{sec:asympt-section-expan-3}, for any
$q=0,\dots,n_{i+1}-m_i-1$ the sum $ \sum\nolimits_{j=0}^{m_i+q} 
 \psi^{cu}(f^jw)$ is bounded from below by
\begin{align}
  -\log\lambda_1
  &\cdot\nonumber
    \#\{0\le j<m_i+q: \phi_1^j(w)\in
    \Sigma_0^*\setminus U\}
    +\sum\nolimits_{k=0}^ia\big(\phi_1^{n_k}w\big)\tau\big(\phi_1^{n_k}w\big)
      \\
  &=\label{eq:estima01}
    \sum\nolimits_{j=0}^{\ell(i,q)}
    \Big(
    \big(-\log\lambda_1\mathbf{1}_{\Sigma_0^*\setminus U}
    +a\cdot\mathbf{1}_U \big)\cdot\tau\Big)\circ P^j (w).
\end{align}
where $\ell(i,q)=i+\sum_{k=0}^{i-1}(n_{k+1}-m_k)+q$.

From the previous
estimates~\eqref{eq:derlog},~\eqref{eq:derlog0},~\eqref{eq:derlog1}
and~\eqref{eq:derlog2}, we can use $0$ as a lower bound for the
summands corresponding to the iterates $j=n_{i+1}, \ldots, m_{i+1}-1$,
so that~\eqref{eq:estima01} with $q=n_{i+1}-m_i-1$ still bounds
$  \sum\nolimits_{j=0}^{n_{i+1}+q} \psi^{cu}(f^jw)$
from below, for $q=0, \dots, m_{i+1}-n_{i+1}-1$.

Thus, for any increasing integer sequence $\omega_n\nearrow\infty$,
for each $n\ge1$ there exists $i=i_n$ so that $m_i\le \omega_i
<m_{i+1}$ and we can estimate
\begin{align}\label{eq:rhsav}
  \frac1{\omega_n}\sum_{j=0}^{\omega_n-1}
  &\psi^{cu}(f^jw)
    \ge
    \left(\frac{n_{i+1}}{\omega_n}\right)\cdot\frac1{n_{i+1}}
    \sum_{j=0}^{n_{i+1}-1}
  \log\|\wedge^2 \big(D\phi_1\mid E^c_{\phi_1^jw}\big)^{-1}\|,
\end{align}
since the difference between the two sums are the negative summands
$-\log\lambda_1$ if $\omega_n<n_{i+1}-1$; or non-negative summands
otherwise. As for the quotient, either $n_{i+1}>\omega_n$ or for
$\leb_\Sigma$-a.e. $w$
\begin{align*}
  \frac{n_{i+1}}{\omega_n}
  \ge
  \frac{n_{i+1}}{n_{i+1}+\tau(P^{\ell(i)+1}w)}
  =
  \left(
  1+
  \frac{\tau(P^{\ell(i)+1}w)/\ell(i)}
  {\frac1{\ell(i)}\sum_{k=0}^{\ell(i)}\tau(P^kw)} 
  \right)^{-1}
  \xrightarrow[n\to+\infty]{}1,
\end{align*}
since we have $\frac1j\tau\circ P^j\to0, \leb_\Sigma$-a.e. together
with~\eqref{eq:intau}. 

Therefore, it is enough to consider limits of the right hand side
average in~\eqref{eq:rhsav}, which from~\eqref{eq:estima01} is bounded
from below by
\begin{align*}
  \lim_{i\to+\infty}
  &\left(
  -\frac{\ell(i)\log\lambda_1}{n_{i+1}}
    \frac1{\ell(i)}\sum_{k=0}^{\ell(i)}
    \big(\mathbf{1}_{\Sigma^*_0\setminus U}\big)(P^kw)
    +
    \frac{\ell(i)}{n_{i+1}}\frac1{\ell(i)}
  \sum_{k=0}^{\ell(i)}\big(a\cdot\tau\cdot\mathbf{1}_{U}\big)(P^kw)
    \right)
  \\
  &=
    \frac1\eta\left(
    -\log\lambda_1(1-\leb_\Sigma(U))
    +\int_U a\cdot\tau\,d\leb_\Sigma
    \right),
\end{align*}
where $\ell(i)=\ell(i,n_{i+1}-m_i-1)$ (as in~\eqref{eq:medias}) and we
have used again~\eqref{eq:intau}.

Finally, from Subsection~\ref{sec:slowing-flow-cc}, we can choose
$0<\zeta_0<1$ so that the last expression between parenthesis becomes
\begin{align*}
    -\log\lambda_1(1-\leb_\Sigma(U))
  +\frac1{1-\zeta_0}\int_U a\cdot\tau\,d\leb_\Sigma
  >0
\end{align*}
This shows that the flow of $Z$ does not satisfy NU2SE, as claimed.

\subsection{Higher co-dimensional versions}
\label{sec:higher-co-dimens}

We can extend the previous construction to higher co-dimensions
similarly to Subsection~\ref{sec:multidimesional-ash}. We restart with
consider $k=\ell+2$ for any given fixed $\ell\ge1$, and extend the
vector fields $Y_i, i=3,4$ as follows
\begin{align*}
  \wh{Y_3}(\varrho \cos\theta, \varrho \sin\theta,w,z) :=
  H^\prime_y(\varrho,z)\cdot(\cos\theta,\sin\theta,0,0)-
  H^\prime_x(\varrho,z)\cdot X;
\end{align*}
for $\varrho\in[-3,3]\times B^\ell\times[-2,2]$, where $B^\ell$ is the
closed unit ball with radius $3$ centered at the origin, similar
to~\eqref{eq:Y2}; and analogously
\begin{align*}
  \wh{Y_4}(\varrho \cos\theta, \varrho \sin\theta,w,z) :=
  H^\prime_y(\varrho,z)\cdot(\cos\theta,\sin\theta,0,0)-
  2\cdot H^\prime_x(\varrho,z)\cdot X.
\end{align*}
It is easy to see that $\diver (\wh{Y_i})=\diver(Y_i)$ and
\begin{align*}
  \|\wedge^{k-1}\big( \wh{\phi}_t\mid_{E^{cu}_w}\big)^{-1}\|
  =
  \|\wedge^{k-2}\big( \wh{\phi}_t\mid_{E^{cu}_w}\big)^{-1}\|
  =
  \cdots
  =
  \|\wedge^2\big( \phi_t\mid_{E^{cu}_w}\big)^{-1}\|;
\end{align*}
where $\wh{\phi}_t$ is the flow of $\wh{Y_i}, i=3,4$, and $\phi_t$ is
the flow of $Y_i, i=3,4$.

We can proceed with the same argument as in the previous subsections
to obtain an attractor $\Lambda$ on the manifold $M$ with dimension
$k+2=\ell+3$ and \emph{asymptotic $k$-sectional hyperbolic}, together
with \emph{absence of non-uniform $p$-sectional-expansion} for all
$2\le p\le k-1$, on a full volume subset of the ambient manifold.

\bibliographystyle{abbrv}
\bibliography{../bibliobase/bibliography}


\end{document}